\documentclass[noinfoline]{imsart}
\usepackage{amsthm,amsmath,amssymb,amsfonts,delarray}
\usepackage[dvipdfmx]{graphicx}
\usepackage{afterpage}
\newtheorem{thm}{Theorem}
\newtheorem{lem}{Lemma}
\newtheorem{prop}{Proposition}
\newtheorem{cor}[thm]{Corollary}

\theoremstyle{definition}
\newtheorem{rem}{Remark}
\begin{document}
\begin{frontmatter}
\title{\Large Non-asymptotic Bayesian Minimax Adaptation}
\runtitle{Non-asymptotic Bayesian Minimax Adaptation}
\begin{aug}
\author{Keisuke Yano$\,^{1}$\corref{}\ead[label=e1]{yano@mist.i.u-tokyo.ac.jp}}
\and
\author{Fumiyasu Komaki$\,^{1,2}$\ead[label=e2]{komaki@mist.i.u-tokyo.ac.jp}}
\affiliation{The University of Tokyo}
\address{$\,^{1}$Department of Mathematical Informatics,
Graduate School of
Information Science and Technology,
The University of Tokyo,
7-3-1 Hongo, Bunkyo-ku, Tokyo 113-8656, Japan}
\printead{e1,e2}\\
\affiliation{The University of Tokyo}
\address{$\,^{2}$RIKEN Brain Science Institute,
2-1 Hirosawa, Wako City,
Saitama 351-0198, Japan}
\runauthor{K. Yano and F. Komaki}
\end{aug}
\begin{abstract}

This paper studies a Bayesian approach to non-asymptotic minimax adaptation in nonparametric estimation.
Estimating an input function on the basis of output functions in a Gaussian white-noise model is discussed.
The input function is assumed to be in a Sobolev ellipsoid with an unknown smoothness and an unknown radius.
Our purpose in this paper is to present a Bayesian approach
attaining minimaxity up to a universal constant without any knowledge 
regarding the smoothness and the radius.
Our Bayesian approach provides not only a rate-exact minimax adaptive estimator in large sample asymptotics
but also a risk bound for the Bayes estimator quantifying the effects of both the smoothness and the ratio of the squared radius to the noise variance,
where the smoothness and the ratio are the key parameters to describe the minimax risk in this model.
Application to non-parametric regression models is also discussed.


\end{abstract}
\end{frontmatter}
\section{Introduction}

Consider estimation of the mean in a Gaussian infinite sequence model.
Let $x=(x_{1},x_{2},\ldots)$ be an observation from $P_{\theta,\varepsilon^{2}}:=\otimes_{i=1}^{\infty}\mathcal{N}(\theta_{i},\varepsilon^{2})$
with an unknown mean $\theta\in l_{2}$  and a known variance $\varepsilon^{2}$.
We assume that $\theta$ is included in a Sobolev ellipsoid
\begin{align}
	\mathcal{E}(\alpha_{0},B):=\bigg{\{}\theta\in l_{2}:\sum_{i=1}^{\infty}i^{2\alpha_{0}}\theta^{2}_{i}\leq B^2\bigg{\}},
\label{def_Sobolev}
\end{align}
where both the smoothness $\alpha_{0}$ and the radius $B$ are unknown.
We measure the performance of an estimator $\hat{\theta}$ of $\theta$
by the normalized mean squared risk 
$R(\theta,\hat{\theta})=
\mathrm{E}_{\theta,\varepsilon^{2}}[||\hat{\theta}(X)-\theta||^{2}]/B^{2}$,
where $\mathrm{E}_{\theta,\varepsilon^{2}}$ is the expectation of $X$ with respect to $P_{\theta,\varepsilon^{2}}$
and
$||v||^{2}:=\sum_{i=1}^{\infty}v_{i}^{2} $ for $v\in l_{2}$.

Estimation in Gaussian infinite sequence models is canonical in the context of nonparametric estimation.
Consider the case in which $\alpha_{0}$ is a positive integer.
With the setting $n=\lfloor 1/\varepsilon^2 \rfloor$,
this estimation is equivalent to
estimation of an input function in a Gaussian white-noise model,
that is, estimating an unknown input function $f$ based on independent and identically distributed (i.i.d.)
output functions $Y_{1}(\cdot),\ldots,Y_{n}(\cdot)$ given by
\begin{align*}
	\mathrm{d}Y_{i}(t)=f(t)\mathrm{d}t+\mathrm{d}W(t), t \in [0,1], \ i=1,\ldots,n,
\end{align*}
where 
$W(\cdot)$ is a standard Brownian motion,
and
$f$ is an $L_{2}[0,1]$ function of which the $L_{2}[0,1]$-norm of the $\alpha_{0}$-th derivative is bounded by $B/\pi^{\alpha_{0}}$.
The correspondence between parameters ($f$ and $\theta$) in Gaussian white-noise and Gaussian infinite sequence models is as follows.
Let $\{\phi_{i}:i=1,\ldots\}$ be the trigonometric series:
\begin{align}
	\begin{split}
	\phi_{1}(t)&:=1,\\
	\phi_{2k}(t)&:=\sqrt{2}\cos(2k\pi t), \, k=1,2,\ldots,\\
	\phi_{2k+1}(t)&:=\sqrt{2}\sin(2k\pi t), \, k=1,2,\ldots.
	\label{eq:trigonometric}
	\end{split}
\end{align}
For $i\in\mathbb{N}$, $\theta_{i}$ corresponds to $\int f(t) \phi_{i}(t)\mathrm{d}t$.
The equivalence of Gaussian white-noise and Gaussian infinite sequence models
can be shown through a sufficiency reduction and through transformation via the trigonometric series;
for the proof including the case in which $\alpha_{0}$ is not an integer,
see Lemma A.3.~in \cite{Tsybakov(2009)}.
Nonparametric regression models are asymptotically equivalent to Gaussian infinite sequence models.
See \cite{BrownandLow(1996)} and Subsection \ref{subsec: application} in this paper.
For comprehensive references,
see \cite{Efromovich(1999),Wasserman(2006),Tsybakov(2009),GineandNickl(2016)}.

The aim of the present paper is 
to develop a Bayesian approach to non-asymptotic minimax adaptation in the Gaussian infinite sequence model.
A non-asymptotically minimax adaptive estimator is defined by an estimator $\hat{\theta}$
for which there exists a positive constant $C$ not depending on $B$ or $\varepsilon$
such that
\begin{align}
\sup_{\theta\in\mathcal{E}(\alpha_{0},B)}R(\theta,\hat{\theta})
\leq C \inf_{\delta}\sup_{\theta\in\mathcal{E}(\alpha_{0},B)} 
R(\theta,\delta) \text{ for any $0<\varepsilon \leq B$}.
	\label{eq:ScaleRatioMinimax}
\end{align}
For this definition of non-asymptotic minimax adaptation,
see p.362 of \cite{BarronBirgeMassart(1999)} and p.212 of \cite{BirgeandMassart(2001)}.
We develop a prior distribution that yields a non-asymptotically minimax adaptive Bayes estimator
and present a risk bound for the Bayes estimator.

Non-asymptotic minimax adaptation is important
since it gives not only the rate of convergence in $\varepsilon$ 
but also a risk bound quantifying the influence of the ratio $\varepsilon/B$ and the smoothness $\alpha_{0}$.
The important point here is 
that only two quantities $\varepsilon/B$ and $\alpha_{0}$ completely determine the difficulty of estimation in the Gaussian infinite sequence model,
that is, the minimax risk over a Sobolev ellipsoid: $ \inf_{\hat{\theta}}  \sup_{\theta\in\mathcal{E}(\alpha_{0},B)}  R  (\theta,\hat{\theta})$.
This is shown by the fact that the minimax risk is invariant whenever the value of $\varepsilon/B$ is unchanged;
For example,
the minimax risk over $\mathcal{E}(\alpha_{0},B)$ with the noise variance $\varepsilon^{2}/100$
is identical to the minimax risk over $\mathcal{E}(\alpha_{0},10B)$ with the noise variance $\varepsilon^{2}$.
In fact,
if we let $\tilde{\theta}=\theta/B$,
$\tilde{\varepsilon}=\varepsilon/B$,
and
$\widetilde{X}=X/B$,
then we have
\begin{align}
	\inf_{\hat{\theta}}\sup_{\theta\in\mathcal{E}(\alpha_{0},B)}R(\theta,\hat{\theta})
	&=\inf_{\hat{\theta}}\sup_{\tilde{\theta}\in\mathcal{E}(\alpha_{0},1)}\mathrm{E}_{\tilde{\theta},\tilde{\varepsilon}^{2}}
	\bigg{[}\sum_{i=1}^{\infty}\{\tilde{\theta}_{i}-\hat{\theta}_{i}(B\widetilde{X})/B\}^{2}\bigg{]}
	\nonumber\\
	&=\inf_{\hat{\theta}} \sup_{\theta\in\mathcal{E}(\alpha_{0},1)} 
	\mathrm{E}_{\theta,\tilde{\varepsilon}^{2}}\left[\sum_{i=1}^{\infty}\{\theta_{i}-\hat{\theta}_{i}(\widetilde{X})\}^{2}\right].
	\label{eq:minimax_invariance}
\end{align}

Developing a Bayesian approach to non-asymptotic minimax adaptation is not straightforward.
Since, if both $B$ and $\alpha_{0}$ are known,
truncation estimator $(1_{1\leq d}X_{1},\ldots, 1_{i\leq d}X_{i},\ldots)$ with $d=(B/\varepsilon)^{1/(2\alpha_{0}+1)}$
attains (\ref{eq:ScaleRatioMinimax}),
one expects that putting a prior distribution on $d$ would attain non-asymptotic minimax adaptation.
However, this idea is not satisfactory as shown in the following.
Consider a prior distribution of the form 
\begin{align*}
\theta \mid D & \sim \bigg{[}\mathop{\otimes}_{i=1}^{D}\mathcal{N}(0,1)\bigg{]}\otimes
	\bigg{[}\mathop{\otimes}_{i=D+1}^{\infty}\mathcal{N}(0,0)\bigg{]}, \\
	D & \sim M,
\end{align*}
where $M$ is a distribution on $\mathbb{N}$.
For the Bayes estimator $\hat{\theta}^{*}$ based on this prior,  we have
\begin{align*}
\sup_{\theta\in\mathcal{E}(\alpha_{0},B)} R(\theta,\hat{\theta}^{*})
\geq \sup_{\theta\in\mathcal{E}(\alpha_{0},B)} \frac{\mathrm{E}_{\theta,\varepsilon^{2}}[(\theta_{1}-\hat{\theta}^{*}_{1}(X))^{2}]}{B^{2}}
\geq  \frac{\varepsilon^4}{(\varepsilon^{2}+1)^{2}} \sup_{\theta\in\mathcal{E}(\alpha_{0},B)}(\theta_{1}/B)^{2},
\end{align*}
where the first inequality follows since the mean squared risk is larger than the coordinate-wise mean squared risk,
and the second inequality follows since $\hat{\theta}_{1}^{*}(X)= \{ \varepsilon^{2} / (\varepsilon^{2}+1) \}  X_{1}$.
The rightmost term in the above inequality is bounded below by $\varepsilon^{4}/ (\varepsilon^{2}+1)$ because $\mathcal{E}(\alpha_{0},B)$ contains $(B,0,0,\ldots)$.
In contrast, 
it follows that the minimax risk goes to 0 as $B$ grows 
since the minimax risk is invariance whenever the value of $B/\varepsilon$ is unchanged
and
since the minimax risk goes to 0 as $\varepsilon$ goes to 0.
Thus, $\hat{\theta}^{*}$ does not attain non-asymptotic minimax adaptation.
Other examples that do not attain non-asymptotic minimax adaptation are presented in Section \ref{sec:Failures}.

We work with a simple prior distribution of the form
\begin{align*}
	\theta \mid (D, K) &\sim \bigg{[}\mathop{\otimes}_{i=1}^{D}\mathcal{N}(0,\varepsilon^{2}D^{2K+1} i^{-(2K+1)} )\bigg{]}\otimes
	\bigg{[}\mathop{\otimes}_{i=D+1}^{\infty}\mathcal{N}(0,0)\bigg{]}, \\
	(D,K) & \sim M \otimes F,
\end{align*}
where $M$ and $F$ are distributions on $\mathbb{N}$.
In Section \ref{sec:nonasymptoticBayesianadaptation},
we show that its Bayes estimator is non-asymptotically adaptive.
The prior is a modification of a sieve prior in the literature; see Subsection \ref{subsec: literature} below.
The modification comes from two principal ideas.
The first idea is to endow $(B/\varepsilon)^2$ with a prior distribution.
Starting from the Gaussian prior distribution
$\otimes_{i=1}^{D}\mathcal{N}(0,Vi^{-2K-1})\otimes\otimes_{i=D+1}^{\infty}\mathcal{N}(0,0)$
given $D$,$V$, and $K$,
we endow $D$, $V$, and $K$ with prior distributions.
Here,
the prior distribution on $V/\varepsilon^2$ corresponds to a prior distribution of $(B/\varepsilon)^2$.
The second idea is to put a prior distribution simultaneously on $D$ and $V$,
focusing on the stochastic behavior 
of the seminorm $\sum_{i=1}^{D}i^{-2K-1}N_{i}^2$ with independent Gaussian random variables $\{N_{i}:i=1,\ldots,D\}$
as shown in Lemma \ref{lem:priormasscondition}.
The second idea also enables us to calculate the posterior distribution easily.


\subsection{Literature review and contributions}
\label{subsec: literature}

There is an extensive literature on asymptotic minimax adaptation in Gaussian infinite sequence models.
Efromoivich and Pinsker \cite{EfromovichandPinsker(1984)} developed an asymptotically minimax adaptive estimator.
Cai, Low, and Zhao \cite{Caietal(2000)} and Cavalier and Tsybakov \cite{CavalierandTsybakov(2001)}
constructed an asymptotically minimax adaptive estimator on the basis of the James--Stein estimator.
There exists a literature from a Bayesian perspective.
Belitser and Ghosal \cite{BelitserandGhosal(2003)} showed that 
putting prior distributions on the hyperparameter $\alpha$ in the Gaussian distribution $\mathrm{G}(\cdot \mid \alpha)$ 
leads to asymptotic minimax adaptation
in the case in which the smoothness is included in a discrete set.
Scricciolo \cite{Scricciolo(2006)} obtained the corresponding result in Bayesian nonparametric density estimation.
Huang \cite{Huang(2004)} removed the assumption on $\alpha_{0}$, incurring the need to pay the price that
a logarithmic factor is added to the rate.
See also 
\cite{Arbeletal(2013),GhosalLembervanderVaart(2008),KnapikvanderVaartvanZanten(2011),Ray(2013)}.
Recently,
the results on rate-exact Bayesian minimax adaptation without any assumption on $\alpha_{0}$
are elegantly established by
\cite{GaoandZhou(2016),HoffmannRousseauSchmidt-Hieber(2015),Johannes_Simoni_Schenk(2016),KnapikSzabovanderVaartvanZanten(2016)}.
However, 
non-asymptotic minimax adaptation implies asymptotic minimax adaptation,
whereas the converse does not hold as shown in Section \ref{sec:Failures}.
To achieve non-asymptotic Bayesian minimax adaptation, further consideration is required in constructing a prior distribution.

Non-asymptotic minimax adaptation has been studied from the viewpoints of model selection and frequentist model averaging.
In nonparametric density estimation,
Birg\'{e} and Massart \cite{BirgeandMassart(1997)} showed that
the estimator based on an analogue of Mallows' $C_{p}$ \cite{Mallows(1973)}
(equivalently, the Akaike Information Criterion (AIC) \cite{Akaike(1973)} and Stein's Unbiased Risk Estimator (SURE) \cite{Stein(1973)}) attains 
non-asymptotic minimax adaptation.
For the corresponding results in nonparametric regression models and in Gaussian infinite sequence models,
see \cite{Baraud(2000),BirgeandMassart(2001)}.
See also \cite{BarronBirgeMassart(1999)} for more general results.
Non-asymptotic minimax adaptation is also attained by frequentist model averaging.
A slight modification of the arguments in 
\cite{DalalyanandSalmon(2012),LeungandBarron(2006)}
shows that
the frequentist model averaging estimator using Mallows' $C_{p}$
has non-asymptotic minimax adaptation as discussed in Section \ref{sec:Failures}.
Yet, the question whether there exists a fully Bayesian approach attaining non-asymptotic minimax adaptation has remained unresolved.
Although there is a connection between the frequentist model averaging estimator and Bayesian model averaging
as discussed in \cite{Hartigan(2002)} (see also the appendix in \cite{LeungandBarron(2006)}),
posterior distributions of $\theta$ are not available in the frequentist model averaging approach.
For this reason, we distinguish frequentist model averaging from Bayesian model averaging.

The prior distribution that we use is a modification of a sieve prior discussed in 
\cite{Arbeletal(2013),Ray(2013),Zhao(2000),ShenandGhosal(2015),ShenandWasserman(2001)}.
Originally, the sieve prior was introduced by Zhao \cite{Zhao(2000)} to resolve some Bayesian nonparametric problems regarding prior masses on the parameter space.
Arbel et al.~\cite{Arbeletal(2013)} showed that the Bayes estimator based on the sieve prior is asymptotically minimax adaptive 
up to a logarithmic factor under the asymptotics as 
$\varepsilon\to 0$.
A practical advantage of sieve priors in Gaussian sequence models is
that an exact sampling from the posterior distributions is performed by a simple acceptance-rejection method.
However, the Bayes estimator based on the original sieve prior in \cite{Zhao(2000)} 
fails non-asymptotic minimax adaptation as discussed in Section \ref{sec:Failures}.
Thus, a non-trivial refinement is necessary.

\subsection{Organization}

The remainder of the paper is organized as follows.
In Section \ref{sec:Failures},
we review several existing estimators from the viewpoint of non-asymptotic adaptation.
In Section \ref{sec:nonasymptoticBayesianadaptation},
a non-asymptotically adaptive Bayes estimator is proposed.
This is the principal part of this study.
Section \ref{sec:numerical} presents numerical experiments.
Numerical experiments report that our Bayesian approach is possibly better than the model selection based estimator
and is comparable to the model averaging estimator.
Section \ref{sec:details_ss} provides proofs of theorems in Section \ref{sec:nonasymptoticBayesianadaptation}.
Some proofs of lemmas and supplementary numerical experiments are provided in appendices.

\section{Non-asymptotic adaptation and existing estimators}
\label{sec:Failures}

In this section,
we review existing estimators from the perspective of non-asymptotic adaptation
with the aim of assisting the reader in understanding non-asymptotic adaptation.

Let us mention a necessary condition for non-asymptotic minimax adaptation ahead.
It is a necessary condition for non-asymptotic minimax adaptation
that the rate of convergence of the minimax risk of an estimator with respect to $\varepsilon/B$
is $(\varepsilon/B)^{4\alpha_{0}/(2\alpha_{0}+1)}$.
In particular,
the rate of convergence of a non-asymptotically minimax adaptive estimator with respect to $1/B$ is $B^{-4\alpha_{0}/(2\alpha_{0}+1)}$.
This is because for each $\alpha_{0}>0$, the asymptotic equality
\begin{align}
	\lim_{B/\varepsilon\to\infty}
	\Big{[}\inf_{\hat{\theta}}\sup_{\theta\in\mathcal{E}(\alpha_{0},B)}
		R(\theta,\hat{\theta})
	\Big{/}  (\varepsilon /B)^{4\alpha_{0}/(2\alpha_{0}+1)}\Big{]}
	= c_{\mathrm{P}}(\alpha_{0})
	\label{eq:PinskerMinimax_scaleratio}
\end{align}
follows from (\ref{eq:minimax_invariance}) and from Pinsker's theorem \cite{Pinsker(1980)} with $B=1$,
where
$c_{\mathrm{P}}(\alpha_{0}):= (2\alpha_{0}+1)^{1/(2\alpha_{0}+1)}\{\alpha_{0}/(\alpha_{0}+1)\}^{4\alpha_{0}/(2\alpha_{0}+1)}$.
Pinsker's theorem states that we have, for each $\alpha_{0}>0$ and each $B>0$,
\begin{align}
	\lim_{\varepsilon\to 0}\Big{[}
\inf_{\hat{\theta}}\sup_{\theta\in\mathcal{E}(\alpha_{0},B)} R(\theta,\hat{\theta})
	\Big{/}
( \varepsilon / B )^{4\alpha_{0}/(2\alpha_{0}+1)}\Big{]}=c_{\mathrm{P}}(\alpha_{0}).
\label{eq:PinskerMinimax}
\end{align}

\subsection{Asymptotic and non-asymptotic minimax adaptation}

Non-asymptotic minimax adaptation implies asymptotic minimax adaptation by definition, whereas the converse does not hold.
Even when $\alpha_{0}$ is known, asymptotic minimaxity in small-$\varepsilon$ asymptotics does not necessarily imply (\ref{eq:ScaleRatioMinimax}).

First, consider the Bayes estimator $\hat{\theta}_{\mathrm{G}(\cdot\mid\alpha)}$
based on the Gaussian distribution
\begin{align*}
\mathrm{G} ( \cdot \mid \alpha )  :=  \otimes_{i=1}^{\infty}  \mathcal{N}  (  0  ,  i^{-2\alpha-1}  ), \ \alpha>0.
\end{align*}
This estimator is shown to achieve asymptotic minimaxity as $\varepsilon\to0$:
\begin{align*}
\lim_{\varepsilon\to 0}
\Big{[}
	\sup_{\theta\in\mathcal{E}(\alpha_{0},B)} R(\theta,\hat{\theta}_{\mathrm{G}(\cdot\mid\alpha_{0})})
\Big{/}
\varepsilon^{4\alpha_{0}/(2\alpha_{0}+1)}
\Big{]}
<\infty;
\end{align*}
see \cite{Freedman(1999),Zhao(2000)}.
Using the necessary condition that we mentioned above,
we show that this estimator does not attain non-asymptotic minimax adaptation.
From the necessary condition
it suffices to show that $\sup_{\theta\in\mathcal{R}}(\alpha_{0},B)$ is bounded below by a positive constant not depending on $B$.
Let $\bar{\theta}$ be an $l_{2}$-vector of which the $i$-th coordinate is $B$ if $i=1$ and $0$ otherwise.
Then, the supremum of $R(\theta,\hat{\theta}_{\mathrm{G}})$ over $\mathcal{E}(\alpha_{0},B)$
is given by
\begin{align*}
	\sup_{\theta\in\mathcal{E}(\alpha_{0},B)}
	R(\theta,\hat{\theta}_{\mathrm{G}})
	\geq R(\bar{\theta},\hat{\theta}_{\mathrm{G}})
	\geq \varepsilon^{4}  / (1+\varepsilon^{2})^2, \ \varepsilon>0,
\end{align*}
and thus the proof is completed.

Second, consider the sieve prior introduced by \cite{Zhao(2000)}:
\begin{align*}
	\mathrm{C}_{M}(\cdot\mid \alpha)
	:=\sum_{d=1}^{\infty}M(d)\bigg{[}
	\mathop{\otimes}_{i=1}^{d}\mathcal{N}(0,i^{-2\alpha-1})\bigg{]}
	\otimes
	\bigg{[}\mathop{\otimes}_{i=d+1}^{\infty}\mathcal{N}(0,0)\bigg{]},
\end{align*}
where $M$ is a probability distribution on $\mathbb{N}$ with $M(d)>0$ for any $d\in\mathbb{N}$.
The Bayes estimator with $\alpha=\alpha_{0}$ is shown to be asymptotically minimax,
and
the Bayes estimator with $\alpha=1/2$ is shown to achieve asymptotic minimax adaptation (up to a logarithmic factor) in a range of the smoothness;
see Theorem 6.1 in \cite{Zhao(2000)} and Proposition 2 in \cite{Arbeletal(2013)}.
We show that the Bayes estimator based on the sieve prior with any $\alpha>0$ is not non-asymptotically minimax adaptive.
Since the first component of $\hat{\theta}_{\mathrm{C}_{M}(\cdot\mid\alpha)}$ for any $\alpha>0$ and any $M$ 
is given by $\{1/(1+\varepsilon^{2})\} X_{1}$,
it follows from the same calculation as that of $\hat{\theta}_{\mathrm{G}(\cdot\mid\alpha)}$ that we have
\begin{align*}
\sup_{\theta\in\mathcal{E}(\alpha_{0},B)} R (\theta , \hat{\theta}_{\mathrm{C}_{M}(\cdot\mid\alpha)})
\geq \varepsilon^{4}  / (1+\varepsilon^{2})^2, \ \varepsilon>0,
\end{align*}
which shows that $\hat{\theta}_{\mathrm{C}_{M}(\cdot\mid\alpha)}$ does not attain non-asymptotic minimax adaptation.

Finally, consider the blockwise James--Stein estimator.
The blockwise James--Stein estimator is shown to achieve asymptotic minimax adaptation;
see \cite{Caietal(2000),CavalierandTsybakov(2001)}.
The construction of the blockwise James--Stein estimator is rather technical and is not presented here.
Note that the blockwise James--Stein estimator is a truncation estimator with $d=\lfloor 1/\varepsilon^{2}\rfloor$,
where
a truncation estimator with dimension $d$ is the $d$-dimensional truncation of a estimator in $l_{2}$.
Any truncation estimator $\hat{\theta}_{(d)}$ with $d$ (possibly depending on $\varepsilon$) does not attain non-asymptotic minimax adaptation;
the proof of this property follows the same line as those of $\hat{\theta}_{\mathrm{G}(\cdot\mid\alpha)}$ and $\hat{\theta}_{\mathrm{C}_{M}(\cdot\mid\alpha)}$
since
the supremum of $R(\theta,\hat{\theta}_{(d)})$ over $\mathcal{E}(\alpha_{0},B)$
is given by
\begin{align*}
	\sup_{\theta\in\mathcal{E}(\alpha_{0},B)}
	R(\theta,\hat{\theta}_{(d)})
	\geq
	\sup_{\theta\in\mathcal{E}(\alpha_{0},B)}
	\sum_{i=d+1}^{\infty}\theta^{2}_{i}/B^2
	\geq (d +1 )^{-2\alpha_{0}}, \ \varepsilon>0.
\end{align*}

The above results are summarized in the following proposition.
\begin{prop}\label{prop: not scale ratio 1}
	The following three hold:
	(i) the Bayes estimator $\hat{\theta}_{\mathrm{G}(\cdot\mid\alpha_{0})}$ does not satisfy (\ref{eq:ScaleRatioMinimax});
	(ii) for any $\alpha>0$, the Bayes estimator $\hat{\theta}_{\mathrm{C}_{M}(\cdot\mid\alpha)}$  does not satisfy (\ref{eq:ScaleRatioMinimax});
	(iii) the blockwise James--Stein estimator does not satisfy (\ref{eq:ScaleRatioMinimax}).
\end{prop}

\subsection{Model selection and model averaging}

We explain that the model selection and model averaging estimators are non-asymptotically adaptive.
For $d \in \mathbb{N}$,
let $\hat{r}_{d}:= -\sum_{i=1}^{d}X_{i}^{2} + 2\varepsilon^2 d$.
Let $\hat{\theta}_{\mathrm{MS}}$ be an estimator of which the $i$-th component is given by $X_{i}1_{i\leq \hat{d}}$,
where 
$\hat{d} \in \mathrm{argmin}_{d\in \mathbb{N}} \hat{r}_{d}$.
Let $\hat{\theta}_{\mathrm{MA},\beta}$ be an estimator 
of which the $i$-th component is given by $\sum_{d=1}^{\infty} w_{d} X_{i}1_{i \leq d}$,
where $w_{d}\propto\exp\{- \beta \hat{r}_{d}/ (2\varepsilon^2) \}$ for $d\in \mathbb{N}$
and $\sum_{d=1}^{\infty}w_{d}=1$.
For simplicity, we assume that $\beta \leq 1/2$.
\begin{prop}[Theorem 1 in \cite{BirgeandMassart(2001)} and Section 7 in \cite{LeungandBarron(2006)}]
\label{prop:model selection and model averaging}
There exist positive constants $C_{1}$ and $C_{2}$ for which the inequalities
\begin{align*}
	&\sup_{\theta \in \mathcal{E}(\alpha_{0},B)} R(\theta , \hat{\theta}_{\mathrm{MS}}) 
	\leq C_{1} (\varepsilon/B)^{4\alpha_{0}/(2\alpha_{0}+1)}, \\
	&\sup_{\theta \in \mathcal{E}(\alpha_{0},B)} R(\theta , \hat{\theta}_{\mathrm{MA},\beta}) 
	\leq C_{2} (\varepsilon/B)^{4\alpha_{0}/(2\alpha_{0}+1)}
\end{align*}
hold,
provided that $\varepsilon/B$ is smaller than one.
Here, $C_{1}$ is a universal constant and $C_{2}$ depends only on $\beta$.
\end{prop}
The proof for the model selection-based estimator follows immediately from Theorem 1 in \cite{BirgeandMassart(2001)}; see also \cite{Yang(2005)}
The proof for the model averaging-based estimator is given in Appendix \ref{Appendix:Proof for Section Failures} for the sake of completeness.

\section{Non-asymptotic Bayesian adaptation}
\label{sec:nonasymptoticBayesianadaptation}

As discussed in the introduction,
we work with
the prior distribution
\begin{align}
	\Pi:=\sum_{k=1}^{\infty}F(k)\mathrm{S}_{M}(\cdot\mid\alpha=k),
\end{align}
where
\begin{align}
	\mathrm{S}_{M}(\cdot \mid \alpha ):=\mathop{\sum}_{d=1}^{\infty} M(d) \mathrm{S}( \cdot \mid d , \alpha)
\end{align}
and
$\mathrm{S}(\cdot\mid d, \alpha)$ is the distribution on $l_{2}$ given by
\begin{align}
	\mathrm{S}(\cdot\mid d,\alpha):=
	\bigg{[}\mathop{\otimes}_{i=1}^{d}\mathcal{N}(0,\varepsilon^{2}d^{2\alpha+1} i^{-(2\alpha+1)} )\bigg{]}\otimes
	\bigg{[}\mathop{\otimes}_{i=d+1}^{\infty}\mathcal{N}(0,0)\bigg{]}.
\end{align}
In the present paper,
$M$ is assumed to be of the form $M(d) \propto \exp (-\eta d)$ with $\eta>0$,
and
$F$ is assumed to be of the form $F(d)\propto \exp (-\gamma d)$ with $\gamma>0$.

\subsection{Principal results}

Theorem \ref{thm:AdaptivePosteriorConcentrationwrtB} 
presents non-asymptotic adaptive
posterior contraction of $\Pi$
and
Corollary \ref{cor:Adaptiveminimaxrate_wrtB} demonstrates 
non-asymptotic adaptation of the Bayes estimator of $\Pi$.
\begin{thm}\label{thm:AdaptivePosteriorConcentrationwrtB}
There exist positive constants $C$ and $c$ depending only on $\alpha_{0}$, $\eta$ of $M$ , and $\gamma$ of $F$
for which
the inequality
\begin{align*}
	\mathrm{E}_{\theta_{0},\varepsilon^{2}}
	\Pi( \|\theta-\theta_{0}\|^2 / B^2 \geq C (\varepsilon/B)^{4\alpha_{0}/(2\alpha_{0}+1)} \mid X )
	\leq
	\exp\{-c(B/\varepsilon)^{2/(2\alpha_{0}+1)}\}
\end{align*}
holds uniformly in $\theta_{0}\in \mathcal{E}(\alpha_{0},B)$
provided that $\varepsilon/B$ is smaller than one.
\end{thm}
The proof of this theorem is given in Section \ref{sec:details_ss}.

\begin{cor}\label{cor:Adaptiveminimaxrate_wrtB}
For every $\alpha_{0}>0$ and every $B>0$,
the Bayes estimator based on $\Pi$ is non-asymptotically adaptive:
there exists a positive constant $C_{3}$ depending only on $\alpha_{0}$, $\eta$ of $M$, $\gamma$ of $F$
for which the inequality
\begin{align*}
	\mathop{\sup}_{\theta\in\mathcal{E}(\alpha_{0},B)}
	R (\theta, \hat{\theta}_{\Pi})
	\leq C_{3} \mathop{\inf}_{\hat{\theta}}\mathop{\sup}_{\theta\in\mathcal{E}(\alpha_{0},B)} R(\theta,\hat{\theta})
	\text{ for any } 0<\varepsilon\leq B
\end{align*}
holds.
\end{cor}

\begin{proof}[Proof of Corollary \ref{cor:Adaptiveminimaxrate_wrtB}]
Take $\theta_{0}$ arbitrarily in $\mathcal{E}(\alpha_{0},B)$.
We show that
\begin{align*}
	\mathop{\sup}_{\theta\in\mathcal{E}(\alpha_{0},B)}
	R (\theta, \hat{\theta}_{\Pi})
	\leq \widetilde{C}_{1} (\varepsilon/B)^{4\alpha_{0}/(2\alpha_{0}+1)}
\end{align*}
for some positive constant $\widetilde{C}_{1}$ not depending on $\varepsilon$ or $B$,
since
it follows
from Theorem 4.9 in \cite{Massart(2007)}
that there exists a universal positive constant $\widetilde{C}_{2}$ 
for which we have
\begin{align*}
	\mathop{\inf}_{\hat{\theta}}
	\mathop{\sup}_{\theta\in\mathcal{E}(\alpha_{0},B)} 
	R(\theta,\hat{\theta}) \geq \widetilde{C}_{2} (\varepsilon/B)^{4\alpha_{0}/(2\alpha_{0}+1)}
	\text{ for any } 0 < \varepsilon \leq B.
\end{align*}

By Jensen's inequality,
we have
\begin{align*}
	\mathrm{E}_{\theta_{0},\varepsilon^{2}}
	||\hat{\theta}_{\Pi} - \theta_{0}||^{2} / B^2
	\leq
	\mathrm{E}_{\theta_{0},\varepsilon^{2}}
	\int ||\theta-\theta_{0}||^{2} / B^2
	\mathrm{d}\Pi(\theta\mid X).
\end{align*}
By Fubini's theorem,
we have
\begin{align*}
        \mathrm{E}_{\theta_{0},\varepsilon^{2}}
	\int ||\theta-\theta_{0}||^{2} / B^2
	\mathrm{d}\Pi(\theta\mid X)
	&=\mathrm{E}_{\theta_{0},\varepsilon^{2}}
	\int_{0}^{\infty}\Pi( \|\theta-\theta_{0}\|^2 / B^2 \geq t \big{|} X)
	\mathrm{d}t
	\nonumber\\
	&=\int_{0}^{\infty} \mathrm{E}_{\theta_{0},\varepsilon^{2}}
	\Pi ( \|\theta-\theta_{0}\|^2 / B^2 \geq t \big{|} X )
	 \mathrm{d}t.
\end{align*}
Taking sufficiently large $C$ depending only on $\alpha_{0}$, $\eta$, and $\gamma$,
and 
dividing $[ 0 , \infty )$ into $[ 0 , C (\varepsilon/B)^{4\alpha_{0}/(2\alpha_{0}+1)} ) $ and
$ [ C (\varepsilon/B)^{4\alpha_{0}/(2\alpha_{0}+1)} , \infty ) $,
Theorem \ref{thm:AdaptivePosteriorConcentrationwrtB} yields
\begin{align*}
	\int_{0}^{\infty} &\mathrm{E}_{\theta_{0},\varepsilon^{2}}
	\Pi( \|\theta-\theta_{0}\|^2 /  B^2 \geq t \big{|} X ) \mathrm{d}t
	\nonumber\\
	&\leq
	C (\varepsilon / B )^{4\alpha_{0}/(2\alpha_{0}+1)}
	+ (C/c) \exp\{ - c ( B/\varepsilon )^{2/(2\alpha_{0}+1)}\},
\end{align*}
where $c$ is the constant in Theorem \ref{thm:AdaptivePosteriorConcentrationwrtB}.
Since constants $C$ and $c$ do not depend on $\theta_{0}$,
the above inequality completes the proof.
\end{proof}

Several remarks are provided in order.

\begin{rem}[Posterior contraction of Gaussain prior distributions]
In Section \ref{sec:Failures},
we showed that the Bayes estimator based on the Gaussian prior $\mathrm{G}(\cdot\mid\alpha)$ does not satisfy (\ref{eq:ScaleRatioMinimax}).
	This prior also does not possess posterior contraction at the rate $(\varepsilon/B)^{4\alpha_{0}/(2\alpha_{0}+1)}$ with respect to $\varepsilon/B$.
Consider $\mathrm{G}(\cdot\mid\alpha)=\otimes_{i=1}^{\infty}\mathcal{N}(0,i^{-2\alpha-1})$.
Let $\varepsilon=1$ and let $\bar{\theta}$ be an $l_{2}$-vector of which the $i$-th coordinate is $B$ if $i=1$ and 0 if otherwise.
For any $\delta>0$, any $C>0$, and $P_{\bar{\theta},1}$-almost all $x$, we have
\begin{align*}
	\mathrm{G}&(||\theta-\bar{\theta}||^{2} /B^2 < CB^{-2\delta}\mid X=x , \alpha=\alpha_{0})\nonumber\\
		&= \mathrm{G}
	\bigg{[}\sum_{i=2}^{\infty}\theta_{i}^{2} + (\theta_{1}-B)^{2} < CB^{2-2\delta} \mid X=x , \alpha=\alpha_{0} \bigg{]}\nonumber\\
	&\leq \mathrm{G} [(\theta_{1}-B)^{2}<CB^{2-2\delta}\mid X=x , \alpha=\alpha_{0} ]\nonumber\\
		&=\mathrm{Pr}
	[B(1-\sqrt{C}B^{-\delta})< (N-x_{1}/\sqrt{2})/\sqrt{2} < B(1+\sqrt{C}B^{-\delta})]\nonumber\\
		&\to 0 \text{ as $B\to\infty$},
\end{align*}
where $N$ is a one-dimensional standard normal random variable.
Thus, by the dominated convergence theorem, we have
\begin{align*}
		\lim_{B\to\infty}\sup_{\theta_{0}\in\mathcal{E}(\alpha_{0},B)}
		\mathrm{E}_{\theta_{0},\varepsilon^{2}}
		\mathrm{G}(||\theta-\theta_{0}||^{2} /B^2 \geq CB^{-2\delta}\mid X , \alpha=\alpha_{0}) =1.
\end{align*}
\end{rem}

\begin{rem}[A further possibility]
\label{rem: Another possibility}
We mention the possibility that the Bayes estimator based on another prior distribution could attain 
non-asymptotic adaptation.
\cite{Szabo_vanderVaart_vanZanten(2013)} 
considered the Bayes estimator $\hat{\theta}_{V}$ based on
the Gaussian scale mixture prior distribution
\begin{align*}
	\int_{0}^{\infty} \int_{0}^{\infty}
	\otimes_{i=1}^{\infty} \mathcal{N}(0, v i^{-2\alpha-1}) \mathrm{d}V(v)\mathrm{d}A(\alpha),
\end{align*}
where $V$ and $A$ are distributions on $[0,\infty)$.
When addressing the Gaussian process prior distribution,
the mixture with respect to the prior variance is often used;
see also \cite{RasmussenandWilliams(2005)}, \cite{vanderVaartandvanZanten(2009)}, and \cite{Suzuki(2012)}.
Although we conjecture that Bayes estimators based on the scale mixtures would be non-asymptotically adaptive 
(see Appendix \ref{Appendix:secondary experiments}),
proving that appears to be challenging.
Our prior distribution $\Pi$ enjoys the discrete structure of a prior distribution of $d$.
A computational advantage of $\Pi$
is that the discrete structure simplifies the calculation of the posterior distribution.
For the explicit form of the posterior distribution, see Appendix \ref{Appendix: explicit form}.
A technical advantage of $\Pi$
is that the calculations of the essential support and the small ball probability might be easier,
as shown in Lemmas \ref{lem:Essentialsupport} and \ref{lem:priormasscondition}.
\end{rem}

\begin{rem}[Non-asymptotic adaptation of $\mathrm{S}_{M}(\cdot\mid\alpha)$ in the undersmooth region]
\label{rem: oversmooth region}
The proof of the principal theorem is based on
the following non-asymptotic adaptation of $\mathrm{S}_{M}(\cdot\mid\alpha)$
in the undersmooth case that $\alpha \geq \alpha_{0}-1/2$.
The proof of the following theorem
is provided in Section \ref{sec:details_ss}.
\begin{thm}
\label{thm:PosteriorConcentrationwrtB}
Assume that $\alpha \geq \alpha_{0}-1/2$.
Then,
there exist positive numbers $C$ and $c$ depending only on $\alpha_{0}$ and $\eta$ of $M$
for which
the inequality
\begin{align*}
	\mathrm{E}_{\theta_{0},\varepsilon^{2}}
		&\mathrm{S}_{M} ( ||\theta-\theta_{0}||^{2} / B^2
	\geq C ( \varepsilon / B )^{4\alpha_{0}/(2\alpha_{0}+1)} \mid X , \alpha )
	\leq \exp\{-c ( B/\varepsilon )^{2/(2\alpha_{0}+1)}\},
\end{align*}
holds uniformly in $\theta_{0}\in \mathcal{E}(\alpha_{0},B)$,
provided that $B/\varepsilon$ is larger than one.
\end{thm}
From Theorem \ref{thm:PosteriorConcentrationwrtB},
the Bayes estimator based on $\mathrm{S}_{M}(\cdot\mid\alpha)$
is non-asymptotically adaptive at least in the undersmooth region where $\alpha \geq \alpha_{0}-1/2$.
\end{rem}

\subsection{Application to non-asymptotically adaptive estimation in nonparametric regression models}
\label{subsec: application}

Our results also construct a non-asymptotically adaptive Bayes estimator in nonparametric regression models.
Consider estimating a regression function $f: [0,1]\to \mathbb{R}$ based on observations $\{Y_{1} , \ldots , Y_{n} \}$ obeying
\begin{align*}
	Y_{i} = f(i/n) + W_{i}, \  i=1,\ldots, n,
\end{align*}
where $W_{i}$s' are ${\it i.i.d.}$ error terms from $\mathcal{N}(0,1)$.
We denote by $\{\phi_{1},\phi_{2},\ldots\}$ the trigonometric series (\ref{eq:trigonometric}).
We assume that $f$ belongs to the periodic Sobolev space $\mathcal{W}(\alpha_{0},B)$ defined as follows:
For $\alpha_{0}>0$ and $B>0$,
\begin{align*}
	\mathcal{W}(\alpha_{0},B):= \bigg{\{} f=\sum_{i=1}^{\infty}\theta_{i}\phi_{i} : \sum_{j=1}^{\infty}a_{j}^2 \theta_{j}^{2} \leq B^2 / \pi^{2\alpha_{0}}  \bigg{\}}  ,
\end{align*}
where $a_{1}=0$ and for $k\in\mathbb{N}$,
\begin{align*}
	a_{2k}= (2k)^{\alpha_{0}} \text{ and } a_{2k+1}= (2k)^{\alpha_{0}}.
\end{align*}
In the cases in which $\alpha_{0}$ is a positive integer,
it follows from the Parseval equality that
$\sum_{j=1}^{\infty}a_{j}^{2}\theta_{j}^{2} = \|f^{(\alpha_{0})}\|^{2}_{L_{2}} / \pi^{2\alpha_{0}}$,
where $\|\cdot\|_{L_{2}}$ is the $L_{2}[0,1]$-norm.

For an arbitrary positive integer $p$ such that $p < n$,
we work with a prior distribution of $f$ using a prior on $\mathbb{R}^{p}$.
Let $\Pi^{(p)}$ be a prior distribution of the form
\begin{align*}
	\Pi^{(p)}
	=\sum_{d=1}^{p}M(d) \sum_{k=1}^{\infty} \mathrm{S}(\cdot\mid d, \alpha=k)
	\bigg{/} \sum_{\tilde{d}=1}^{p}M(\tilde{d}).
\end{align*}
We identify $\Pi^{(p)}$ as a prior distribution of $f$ via the transformation 
$(\theta_{1},\ldots,\theta_{p}) \to \sum_{i=1}^{p} \theta_{i}\phi_{i}$.
Let $\hat{f}_{\Pi^{(p)}}$ be the Bayes estimator based on $\Pi^{(p)}$.

The following corollary provides a non-asymptotic risk bound for the Bayes estimator $\hat{f}_{\Pi^{(p)}}$.
Let $R_{n}(f,\hat{f})$ be 
the risk of an estimator $\hat{f}$ defined by
\begin{align*}
	R_{n}(f,\hat{f}):= \mathrm{E}_{f,n}\| f - \hat{f} \|^2_{L_2}/ B^2,
\end{align*}
where 
$\|\cdot\|_{L_2}$ is the $L_{2}[0,1]$ norm,
and
$\mathrm{E}_{f,n}$ is the expectation with respect to the distribution of  $\{Y_{1} , \ldots , Y_{n}\}$ with true regression function $f$.
Let $\tau(p)$ be the approximation error
\begin{align*}
\tau(p):= \sup_{f\in\mathcal{W}(\alpha_{0},B)}\inf_{\theta\in \mathbb{R}^{p} } \left\|f-\sum_{i=1}^{p}\theta_{i}\phi_{i}\right\|^2_{L_{2}}.
\end{align*}
\begin{cor}
\label{cor:finite}
There exists a positive constant $C_{4}$ depending only on $\alpha_{0}$, $\eta$ of $M$, and $\gamma$ of $F$
for which
the Bayes estimator $\hat{f}_{\Pi^{(p)}}$ based on $\Pi^{(p)}$ satisfies
\begin{align}
	\sup_{ f \in \mathcal{W}(\alpha_{0},B)} 
	R_{n} (f , \hat{f}_{\Pi^{(p)}})
	\leq C_{4}[\min\{ p , (n B^2 )^{1/(2\alpha_{0}+1)} \} / n + \tau(p) ]/B^2 ,
\label{eq:cor:finite}
\end{align}
provided that $nB$ is larger than one.
\end{cor}
The proof is a simple extension of that of Theorem \ref{thm:AdaptivePosteriorConcentrationwrtB} and
is given in Subsection\ref{subsec:proofofcor}.

The implication of this corollary is that 
$\hat{f}_{\Pi^{(n-1)}}$ is non-asymptotically adaptive, provided that $n^{\alpha_{0}}\geq B$.
From Corollary \ref{cor:finite} with $p=n-1$, and from the bound that $\tau(p)\leq B^2 p^{-2\alpha_{0}}$,
there exists a positive constant $C_{5}$ not depending on $n$ or $B$ such that
\begin{align*}
\sup_{f\in \mathcal{W}(\alpha_{0},B)} R_{n} (f , \hat{f}_{\Pi^{(n-1)}}) 
\leq C_{5} (n B^{2})^{-2\alpha_{0}/ (2\alpha_{0}+1)},
\end{align*}
provided that $n^{\alpha_{0}}\geq B$.
From Theorem 4.9 in \cite{Massart(2007)},
there exists a positive constant $C_{6}$ depending only on $\alpha_{0}$ for which the inequality
\begin{align*}
	C_{6}\max_{1\leq p < n}\min\{ p^{-2\alpha_{0}},p/(nB^2) \}
 	\leq \inf_{\hat{f}} \sup_{f \in \mathcal{W}(\alpha_{0},B)} R_{n} (f,\hat{f})
\end{align*}
holds.
Thus, 
there exists a positive constant $C_{7}$ not depending on $n$ or $B$ such that
\begin{align*}
	\sup_{f\in \mathcal{W}(\alpha_{0},B)} R_{n}(f, \hat{f}_{\Pi^{(n-1)}})
	\leq C_{7} \inf_{\hat{f}} \sup_{f \in \mathcal{W}(\alpha_{0},B)} R_{n} (f,\hat{f}),
\end{align*}
provided that $n^{\alpha_{0}}\geq B$.

\section{Numerical experiments}
\label{sec:numerical}

In this section, 
we present numerical experiments 
focusing on the performance comparison of non-asymptotically adaptive estimators
in low-$(B/\varepsilon)$ settings.
The other comparisons
including the comparison between non-asymptotically adaptive estimators 
and estimators not satisfying (\ref{eq:ScaleRatioMinimax})
are provided in Appendix \ref{Appendix:secondary experiments}.

The numerical experiments are intended to compare non-asymptotically adaptive estimators.
The following three estimators are compared:
\begin{itemize}
\item the Bayes estimator $\hat{\theta}_{\Pi}$ based on $\Pi$ with $\eta=2$ and $\gamma=2$;
\item the model averaging estimator $\hat{\theta}_{\mathrm{MA},1/2}$ with $\beta=1/2$;
\item the model selection-based estimator $\hat{\theta}_{\mathrm{MS}}$.
\end{itemize}
The numerical experiments are conducted using the $p=100$-dimensional truncation.
The noise variance $\varepsilon^2$ is fixed to one and the volume $B^2$ is varied in $\{1,2,3,4,5\}$.
Losses at two parameter values are used for comparison.
The following parameter values are used:
\begin{itemize}
	\item $\theta^{(1)}_{i}:=Bi^{-0.52}/\sqrt{100}$ for $i\in\mathbb{N}$;
	\item $\theta^{(2)}_{1}:=B$ and $\theta^{(2)}_{i}:=0$ for $i\geq 2$.
\end{itemize}
Note that
$\theta^{(1)}$ is included in $\mathcal{E}(\alpha_{0},B)$ for any $0<\alpha_{0}<0.014$
and
is not included in $\mathcal{E}(\alpha_{0},B)$ for any $\alpha_{0}>0.015$.
Note also that
$\theta^{(2)}$ is included in $\mathcal{E}(\alpha_{0},B)$ for any $\alpha_{0}>0$.

\begin{figure}[h]
\begin{center}
\includegraphics[width=0.60\hsize]{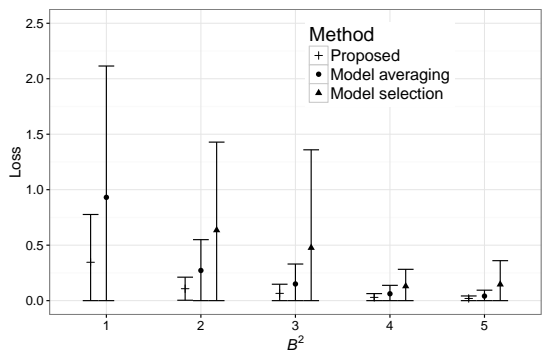}
\caption{
Means of losses with error bars at $\theta=\theta^{(1)}$ in cases with $B^2=1,2,3,4,5$.
The means and error bars of the model selection based estimator at $B^{2}=1$ are omitted because
they are outside the range $[0,2.5]$.
}
\label{Loss_1}
\end{center}
\end{figure}
\begin{figure}[h]
\begin{center}
\includegraphics[width=0.60\hsize]{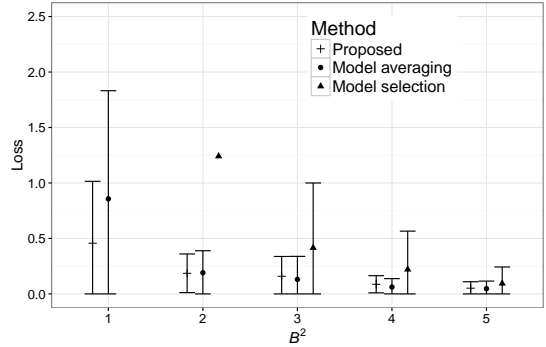}
\caption{
Means of losses with error bars at $\theta=\theta^{(2)}$ in cases with $B^2=1,2,3,4,5$.
The means and error bars of the model selection based estimator at $B^{2}=1$ are omitted because
they are outside the range $[0,2.5]$.
The error bars of the model selection based estimator at $B^{2}=2$ are omitted becasue
the upper bar is outside the range $[0,2.5]$.
}
\label{Loss_2}
\end{center}
\end{figure}

The results are presented in Figures \ref{Loss_1} and \ref{Loss_2}.
At each $B$, means (with standard deviations) of the proposed Bayes estimator $\hat{\theta}_{\Pi}$,
the model averaging estimator $\hat{\theta}_{\mathrm{MA},1/2}$, and 
the model selection based estimator $\hat{\theta}_{\mathrm{MS}}$
are plotted side-by-side.
The proposed estimator $\hat{\theta}_{\Pi}$ is abbreviated by ``Proposed;"
the model averaging estimator $\hat{\theta}_{\mathrm{MA},1/2}$ is abbreviated by ``Model averaging;"
the model selection-based estimator $\hat{\theta}_{\mathrm{MS}}$ is abbreviated by ``Model selection."
In each plot,
the lower limit used in the error bar is calculated as the maximum of zero and the mean minus the standard deviation.
We omit the values outside the range $[0,2.5]$.

These figures indicate that 
the proposed Bayes estimator outperforms the model selection based estimator.
This outperformance does not depend on $\alpha_{0}$ and $B$.
Figure \ref{Loss_1} indicates that
the proposed Bayes estimator outperforms the model averaging estimator,
while
Figure \ref{Loss_2} indicates that
the proposed Bayes estimator underperforms the model averaging estimator.
However, even in Figure \ref{Loss_2},
when $B$ is small,
the performance of the proposed Bayes estimator is comparable (or possibly superior) to 
that of the model averaging estimator.
Compared to the model averaging estimator,
our approach directly puts a prior distribution on the scale of the parameter,
which seems to present the better outcome when $B$ is small.

\section{Proof for Section \ref{sec:nonasymptoticBayesianadaptation}}
\label{sec:details_ss}

The proofs follow the standard arguments in the Bayesian nonparametric literature
\cite{BarronSchervishWasserman(1999),Ghosal_Ghosh_vanderVaart(2000),ShenandWasserman(2001)}.
The essential difference appears in the prior mass condition with respect to $\varepsilon/B$
under which the prior puts a sufficient mass on the neighbors around the true parameter with respect to $\varepsilon/B$;
see Lemma \ref{lem:priormasscondition}.

The organization of this section is as follows.
In Subsection \ref{subsec:lemmas}, we prepare some lemmas to be used.
In Subsection \ref{subsec:mainproof}, we present the proof of Theorem \ref{thm:PosteriorConcentrationwrtB}.
In Subsection \ref{subsec:subproof}, we present the proof of Theorem \ref{thm:AdaptivePosteriorConcentrationwrtB}.

\subsection{Lemmas}
\label{subsec:lemmas}

In this subsection,
we present our lemmas.
The proofs of the lemmas are provided in Appendix \ref{subsec:proofs_lemmas}.
Note that 
\begin{align*}
	&\{\theta: \|\theta-\theta_{0}\|^2 / B^2 \geq C (\varepsilon / B)^{4\alpha_{0}/(2\alpha_{0}+1)} \}
	=
	\{\theta: \|\theta-\theta_{0}\|^2 / \varepsilon^2 \geq C (B/\varepsilon)^{2/(2\alpha_{0}+1)} \}.
\end{align*}

The first lemma provides the essential support of $\mathrm{S}_{M,\alpha}$.
For a constant $c_{1}>0$ and $\theta_{0}\in\mathcal{E}(\alpha_{0},B)$,
let
\begin{align*}
	E_{c_{1}}(\theta_{0}):=
	\left\{\theta\in l_{2}:\mathop{\sum}_{i>\lfloor c_{1}(B/\varepsilon)^{2/(2\alpha_{0}+1)} \rfloor}
		(\theta_{i}-\theta_{0,i})^{2}/\varepsilon^{2}
		\leq (B/\varepsilon)^{2/(2\alpha_{0}+1)}
	\right\}.
\end{align*}
\begin{lem}[Essential support of the prior]
	\label{lem:Essentialsupport}
	For any $\alpha > 0 $ and any $c_{1}>1$,
	the inequality
	\begin{align*}
		\mathrm{S}_{M}
		(E^{\mathrm{c}}_{c_{1}}(\theta_{0}) \mid \alpha )
		\leq \exp\{- \eta (c_{1} -1 )(B/\varepsilon)^{2/(2\alpha_{0}+1)}\}
	\end{align*}
	holds uniformly in $\theta_{0}\in\mathcal{E}(\alpha_{0},B)$.
\end{lem}

The second and third lemmas provide the complexity of the interest space and the existence of test sequences.
For a positive integer $C>0$ and a constant $c_{1}>0$,
we divide $\{\theta\in l_{2}:\|\theta-\theta_{0}\|^{2}/\varepsilon^{2}\geq C(B/\varepsilon)^{2/(2\alpha_{0}+1)}\}$
as
\begin{align*}
	\{\theta:& ||\theta-\theta_{0}||^{2} / \varepsilon^{2}  \geq C ( B / \varepsilon)^{2/(2\alpha_{0}+1)}\} \nonumber\\
&= \mathop{\cup}_{j=C}^{\infty}R(j;c_{1})
	\cup
	[\{\theta: ||\theta-\theta_{0}||^{2} / \varepsilon^{2} \geq C ( B / \varepsilon )^{2/(2\alpha_{0}+1)}\}\cap E^{\mathrm{c}}_{c_{1}}(\theta_{0})],
\end{align*}
where for $j=C,C+1,\ldots$,
\begin{align*}
R(j;c_{1})
:=\{\theta\in E_{c_{1}}(\theta_{0}): (j+1)( B / \varepsilon)^{2/(2\alpha_{0}+1)}
	> ||\theta-\theta_{0}||^{2} / \varepsilon^{2} \geq 
	j ( B / \varepsilon)^{2/(2\alpha_{0}+1)} \}.
\end{align*}
For $j=C,C+1,\ldots$,
let $N(j;c_{1})$ be the $(\varepsilon/8)\sqrt{j ( B / \varepsilon )^{2/(2\alpha_{0}+1)}}$-covering number
with respect to $||\cdot||$ of $R(j;c_{1})$.
\begin{lem}[Covering number of $R(j;c_{1})$; cf.~Proposition A.1. in \cite{GaoandZhou(2016)}]
\label{lem:coveringnumber}
For each $j=C,C+1,\ldots$ and every $c_{1}>0$,
$\log(N(j;c_{1}))$ is bounded above by $2c_{1}(B/\varepsilon)^{2/(2\alpha_{0}+1)}$.
\end{lem}
\begin{lem}[Existence of test sequences; cf.~Lemma 5 in \cite{GhosalandvanderVaart(2007)}]
\label{lem:testcondition}
	Let $j$ be any positive integer.
	Let $\theta_{0}$ be in $\mathcal{E}(\alpha_{0},B)$.
	Let $\bar{\theta}_{(j)}$ be any $l_{2}$-vector such that $||\bar{\theta}_{(j)}-\theta_{0}||^{2}/\varepsilon^{2} \geq j (B/\varepsilon)^{2/(2\alpha_{0}+1)}$.
	Let $\psi_{(j)}(X):=1_{||X-\bar{\theta}_{(j)}||^{2}<||X-\theta_{0}||^{2}}$.
	Then,
	the inequalities
	\begin{align*}
		\mathrm{E}_{\theta_{0},\varepsilon^{2}}[\psi_{(j)}(X)] \leq \exp \{-(j/8)(B/\varepsilon)^{2/(2\alpha_{0}+1)}\}
	\end{align*}
	and
	\begin{align*}
		\mathop{\sup}_{\theta: \|\theta-\bar{\theta}_{(j)}\| \leq \|\bar{\theta}_{(j)}-\theta_{0}\|/4 }
		\mathrm{E}_{\theta,\varepsilon^{2}}[1-\psi_{(j)}(X)]\leq 
		\exp\{-(j/32)(B/\varepsilon)^{2/(2\alpha_{0}+1)} \}
	\end{align*}
	hold.
\end{lem}

The fourth lemma is the prior mass condition.
\begin{lem}[Prior mass condition]
	\label{lem:priormasscondition}
	Assume that $\alpha\geq \alpha_{0}-1/2$.
	There exists a positive constant $c_{2}$ depending only on $\alpha_{0}$ and $\eta$ of $M$ for which
	the inequality
	\begin{align*}
		\mathrm{S}_{M}(\theta: \|\theta-\theta_{0}\|^{2} / \varepsilon^{2} \leq 2( B / \varepsilon)^{2/(2\alpha_{0}+1)} \mid \alpha)
		\geq \exp\{-c_{2}(B/\varepsilon)^{2/(2\alpha_{0}+1)}\}
	\end{align*}
	holds uniformly in $\theta_{0} \in \mathcal{E}(\alpha_{0},B)$,
	provided that $\varepsilon/B$ is smaller than one.
\end{lem}

The fifth lemma ensures a high probability set on which the likelihood ratio of the marginal distribution and the true distribution is bounded below.
We denote the restriction of $\mathrm{S}_{M}(\cdot\mid\alpha)$ 
onto $\{\theta:\|\theta-\theta_{0}\|^{2}/\varepsilon^{2}\leq 2(B/\varepsilon)^{2/(2\alpha_{0}+1)}\}$
by $\widetilde{\mathrm{S}}_{M}(\cdot\mid\alpha)$:
\begin{align*}
	\widetilde{\mathrm{S}}_{M}(A\mid\alpha):=\frac{\mathrm{S}_{M}(A\mid\alpha)}
	{\mathrm{S}_{M} (\{\theta:\|\theta-\theta_{0}\|^{2}/\varepsilon^{2}\leq 2(B/\varepsilon)^{2/(2\alpha_{0}+1)} \} \mid\alpha)}
\end{align*}
for a Borel set $A$ in $l_{2}\cap \{\theta:\|\theta-\theta_{0}\|^{2}/\varepsilon^{2}\leq 2(B/\varepsilon)^{2/(2\alpha_{0}+1)}\}$.
Let 
\begin{align*}
	H(\theta_{0}):=\bigg{\{}X:\log \int \frac{\mathrm{d}P_{\theta,\varepsilon^{2}}}{\mathrm{d}P_{\theta_{0},\varepsilon^{2}}}(X)\mathrm{d}\widetilde{\mathrm{S}}_{M}(\theta\mid\alpha) 
	\geq -2 ( B / \varepsilon)^{2/(2\alpha_{0}+1)}\bigg{\}}.
\end{align*}

\begin{lem}
	\label{lem:highprobabilityset}
	For every $\theta_{0}\in\mathcal{E}(\alpha_{0},B)$,
	the inequality
	\begin{align*}
		\mathrm{E}_{\theta_{0},\varepsilon^{2}}[1_{H^{\mathrm{c}}(\theta_{0})}(X)]
		\leq \exp\{-(1/2)( B / \varepsilon)^{2/(2\alpha_{0}+1)}\}
	\end{align*}
	holds.
\end{lem}

\subsection{Proof of Theorem \ref{thm:PosteriorConcentrationwrtB}}
\label{subsec:mainproof}

The proof assumes that $C$ is a positive integer.
If $C$ is not an integer,
we replace $C$ with $\lfloor C\rfloor$.
The values of $C$ in Theorem \ref{thm:PosteriorConcentrationwrtB} and $c_{1}$ in Lemma \ref{lem:Essentialsupport} are provided in (\ref{condition:C}) below.
Take $\theta_{0}$ arbitrarily in $\mathcal{E}(\alpha_{0},B)$.
Recall the equality
\begin{align*}
	\{\theta: \|\theta-\theta_{0}\|^2 / B^2 \geq C (\varepsilon / B)^{4\alpha_{0}/(2\alpha_{0}+1)} \}
	=
	\{\theta: \|\theta-\theta_{0}\|^2 / \varepsilon^2 \geq C (B/\varepsilon)^{2/(2\alpha_{0}+1)} \}.
\end{align*}

The expectation of the tail probability of the posterior is divided as follows:
\begin{align}
&\mathrm{E}_{\theta_{0},\varepsilon^{2}}
[\mathrm{S}_{M}( \|\theta-\theta_{0}\|^{2} / \varepsilon^{2}
\geq C ( B / \varepsilon )^{2/(2\alpha_{0}+1)} \mid X , \alpha)]
\nonumber\\
&\quad=\mathrm{E}_{\theta_{0},\varepsilon^{2}}
[1_{H(\theta_{0})}(X) \mathrm{S}_{M}(\|\theta-\theta_{0}\|^{2} / \varepsilon^{2} \geq 
C ( B / \varepsilon)^{2/(2\alpha_{0}+1)} \mid X , \alpha )]
\nonumber\\
&\quad\quad+\mathrm{E}_{\theta_{0},\varepsilon^{2}}
[1_{H^{\mathrm{c}}(\theta_{0})}(X)\mathrm{S}_{M}( \|\theta-\theta_{0}\|^{2} / \varepsilon^{2} 
\geq C ( B / \varepsilon )^{2/(2\alpha_{0}+1)} \mid X , \alpha )].
\label{eq:decomp}
\end{align}
From Lemma \ref{lem:highprobabilityset}, and because the probability is bounded above by one,
the latter term on the right hand side of (\ref{eq:decomp}) is bounded as follows:
\begin{align}
\mathrm{E}_{\theta_{0},\varepsilon^{2}}
&[1_{H^{\mathrm{c}}(\theta_{0})}(X)\mathrm{S}_{M}( \|\theta-\theta_{0}\|^{2} / \varepsilon^{2}
\geq C (B / \varepsilon)^{2/(2\alpha_{0}+1)} \mid X , \alpha )]
\nonumber\\
&\leq
\exp\{- (1/2) ( B / \varepsilon )^{2/(2\alpha_{0}+1)}\}.
\label{eq:decomp_firstterm}
\end{align}
We next bound the former term in the right-hand side of (\ref{eq:decomp}).

From Bayes' theorem, we have
	\begin{align}
		&\mathrm{E}_{\theta_{0},\varepsilon^{2}}
		[1_{H(\theta_{0})}(X)
		\mathrm{S}_{M} ( ||\theta-\theta_{0}||^{2} / \varepsilon^{2} 
		\geq C ( B / \varepsilon )^{2/(2\alpha_{0}+1)} \mid X , \alpha
		)]\nonumber\\
		&\quad=
		\mathrm{E}_{\theta_{0},\varepsilon^{2}}
		\bigg{[}1_{H(\theta_{0})}(X)
		\frac{\int_{ ||\theta-\theta_{0}||^{2} / \varepsilon^{2} \geq C (B/\varepsilon)^{2/(2\alpha_{0}+1)}}
			\frac{\mathrm{d}P_{\theta,\varepsilon^{2}}}{\mathrm{d}P_{\theta_{0}},\varepsilon^{2} }(X)  \mathrm{d}\mathrm{S}_{M}(\theta \mid \alpha) }
			{\int \frac{\mathrm{d}P_{\theta,\varepsilon^{2}}}{\mathrm{d}P_{\theta_{0},\varepsilon^{2}}}(X)  \mathrm{d}\mathrm{S}_{M}(\theta \mid \alpha) } 
		\bigg{]}.
		\label{eq:firstterm_Bayes}
	\end{align}
Consider the numerator
 \[\int_{ \|\theta-\theta_{0}\|^{2} / \varepsilon^{2} \geq C ( B / \varepsilon)^{2/(2\alpha_{0}+1)}} \frac{\mathrm{d}P_{\theta,\varepsilon^{2}}}{\mathrm{d}P_{\theta_{0},\varepsilon^{2}}}(X)
\mathrm{d}\mathrm{S}_{M}(\theta\mid\alpha).\]
Letting $\{ \bar{\theta}_{(j,k)}: k=1,\ldots, N(j;c_{1})\}$
be an $ ( \varepsilon / 8 ) \sqrt{j ( B / \varepsilon)^{2/(2\alpha_{0}+1)}}$-net of $R(j;c_{1})$,
Lemma \ref{lem:testcondition} yields sequences of measurable functions $\psi_{j,k}$ 
	such that for each $k$, we have
	\begin{align}
	\mathrm{E}_{\theta_{0},\varepsilon^{2}}[\psi_{j,k}(X)]
	\leq \exp\{-(j/8)(B/\varepsilon)^{2/(2\alpha_{0}+1)}\}
	\label{type1_error}
	\end{align}
	and
	\begin{align}
	\sup_{\theta: \|\theta-\bar{\theta}_{(j,k)}\| < (\varepsilon/4)
	\sqrt{j (B/\varepsilon)^{2/(2\alpha_{0}+1)}} }
	\mathrm{E}_{\theta,\varepsilon^{2}}[1-\psi_{j,k}(X)]
	\leq \exp\{-(j/32)(B/\varepsilon)^{2/(2\alpha_{0}+1)}\}.
	\label{balluniform_type2_error}
	\end{align}
Letting $U(\bar{\theta}_{(j,k)})$ be the $(\varepsilon/8) \sqrt{j(B/\varepsilon)^{2/(2\alpha_{0}+1)}}$-ball 
around $\bar{\theta}_{(j,k)}$,
and
using the sequences $\{\psi_{j,k}\}$ and the balls $\{U(\bar{\theta}_{(j,k)})\}$,
we have, for $X\in H(\theta_{0})$,
\begin{align*}
	&\int_{ \|\theta-\theta_{0}\|^{2} / \varepsilon^{2} \geq C ( B / \varepsilon )^{2/(2\alpha_{0}+1)}}
\frac{\mathrm{d}P_{\theta,\varepsilon^{2}}}{\mathrm{d}P_{\theta_{0},\varepsilon^{2}}}(X)
\mathrm{d}\mathrm{S}_{M}(\theta \mid \alpha)
\nonumber\\
&\quad\leq\sum_{j=C}^{\infty}\sum_{k=1}^{N(j;c_{1})}\int_{U(\bar{\theta}_{(j,k)})}(1-\psi_{j,k}(X))\frac{\mathrm{d}P_{\theta,\varepsilon^{2}}}{\mathrm{d}P_{\theta_{0},\varepsilon^{2}}}(X)\mathrm{d}\mathrm{S}_{M}(\theta \mid \alpha)
\nonumber\\
&\quad\quad+\sum_{j=C}^{\infty}\sum_{k=1}^{N(j;c_{1})}\int_{U(\bar{\theta}_{(j,k)})}\psi_{j,k}(X)\frac{\mathrm{d}P_{\theta,\varepsilon^{2}}}{\mathrm{d}P_{\theta_{0},\varepsilon^{2}}}(X)\mathrm{d}\mathrm{S}_{M}(\theta \mid \alpha)
\nonumber\\
&\quad\quad+\int_{E^{\mathrm{c}}_{c_{1}}(\theta_{0})}
\frac{\mathrm{d}P_{\theta,\varepsilon^{2}}}{\mathrm{d}P_{\theta_{0},\varepsilon^{2}}}(X)
\mathrm{d}\mathrm{S}_{M}(\theta \mid \alpha).
\end{align*}
From the above inequality,
it follows that
\begin{align}
\mathrm{E}_{\theta_{0},\varepsilon^{2}}
[1_{H(\theta_{0})}(X)
\mathrm{S}_{M} ( \|\theta-\theta_{0}\|^{2} / \varepsilon^{2} \geq C ( B / \varepsilon )^{2/(2\alpha_{0}+1)}
\mid X , \alpha
)]
\leq T_{1}+T_{2}+T_{3},
\label{eq:decomp_three}
\end{align}
where
\begin{align*}
	T_{1}&:=\mathrm{E}_{\theta_{0},\varepsilon^{2}}
\left[1_{H(\theta_{0})}\frac{ \sum_{j=C}^{\infty}\sum_{k=1}^{N(j;c_{1})}\int_{U(\bar{\theta}_{(j,k)})}(1-\psi_{j,k})
\frac{\mathrm{d}P_{\theta,\varepsilon^{2}}}{\mathrm{d}P_{\theta_{0},\varepsilon^{2}}}
\mathrm{d}\mathrm{S}_{M}(\theta \mid \alpha) }
{\int \frac{\mathrm{d}P_{\theta,\varepsilon^{2}}}{\mathrm{d}P_{\theta_{0},\varepsilon^{2}}}
\mathrm{d}\mathrm{S}_{M}(\theta \mid \alpha )  }\right],
\end{align*}
\begin{align*}
T_{2}&:=\mathrm{E}_{\theta_{0},\varepsilon^{2}}
\left[1_{H(\theta_{0})}\frac{ \sum_{j=C}^{\infty}\sum_{k=1}^{N(j;c_{1})}\int_{U(\bar{\theta}_{(j,k)})}\psi_{j,k}\frac{\mathrm{d}P_{\theta,\varepsilon^{2}}}{\mathrm{d}P_{\theta_{0},\varepsilon^{2}}}
\mathrm{d}\mathrm{S}_{M}(\theta \mid \alpha)  }
{\int \frac{\mathrm{d}P_{\theta,\varepsilon^{2}}}{\mathrm{d}P_{\theta_{0},\varepsilon^{2}}}
\mathrm{d}\mathrm{S}_{M}(\theta \mid \alpha) }\right],
\end{align*}
and
\begin{align*}
T_{3}&:=\mathrm{E}_{\theta_{0},\varepsilon^{2}}
\left[1_{H(\theta_{0})}(X) \frac{ \int_{E^{\mathrm{c}}_{c_{1}}(\theta_{0})}\frac{\mathrm{d}P_{\theta,\varepsilon^{2}}}{\mathrm{d}P_{\theta_{0},\varepsilon^{2}}}(X)\mathrm{d}\mathrm{S}_{M}(\theta \mid \alpha)  }
{\int \frac{\mathrm{d}P_{\theta,\varepsilon^{2}}}{\mathrm{d}P_{\theta_{0},\varepsilon^{2}}}(X)
\mathrm{d}\mathrm{S}_{M}(\theta \mid \alpha) }\right] .
\end{align*}
Providing upper bounds on $T_{1}$, $T_{2}$, and $T_{3}$ will complete the proof.

Consider an upper bound on $T_{1}$ in (\ref{eq:decomp_three}).
In bounding $T_{1}$,
we use the following lower bound on
$\int \{ \mathrm{d}P_{\theta,\varepsilon^{2}} / \mathrm{d}P_{\theta_{0},\varepsilon^{2}}\}(X)
\mathrm{d}\mathrm{S}_{M}(\theta \mid \alpha)$.
From the definition of $H(\theta_{0})$ and from Lemma \ref{lem:priormasscondition},
for $X\in H(\theta_{0})$,
we have
\begin{align}
\int &\frac{\mathrm{d}P_{\theta,\varepsilon^{2}}}{\mathrm{d}P_{\theta_{0},\varepsilon^{2}}}(X)
\mathrm{d}\mathrm{S}_{M}(\theta \mid \alpha)
\nonumber\\
&\geq
\int_{ \|\theta-\theta_{0}\|^{2} / \varepsilon^{2} \leq 2 ( B / \varepsilon)^{2/(2\alpha_{0}+1)}} \frac{\mathrm{d}P_{\theta,\varepsilon^{2}}}{\mathrm{d}P_{\theta_{0},\varepsilon^{2}}}(X)
\mathrm{d}\mathrm{S}_{M}(\theta \mid \alpha)
\nonumber\\
&=\mathrm{S}_{M}(\|\theta-\theta_{0}\|^{2} / \varepsilon^{2}\leq 2( B  / \varepsilon )^{2/(2\alpha_{0}+1)})
\int \frac{\mathrm{d}P_{\theta,\varepsilon^{2}}}{\mathrm{d}P_{\theta_{0},\varepsilon^{2}}}(X)
\mathrm{d}\widetilde{\mathrm{S}}_{M}(\theta \mid \alpha)
\nonumber\\
&\geq
\exp\{-(c_{2}+2)(B/\varepsilon)^{2/(2\alpha_{0}+1)}\}.
\label{eq:lowerbound_denom}
\end{align}
From the above inequality, from  Fubini's theorem, from Lemmas \ref{lem:coveringnumber} and \ref{lem:testcondition},
and
from the inequality that
$1-\exp\{- 1 / 32 \} > 1/\mathrm{e}^{4}$,
we have
\begin{align}
T_{1}&\leq \exp\{(c_{2}+2) ( B / \varepsilon)^{2 / (2\alpha_{0}+1)} \}
	\mathrm{E}_{\theta_{0},\varepsilon^{2}}
	\bigg{[}\sum_{j=C}^{\infty}\sum_{k=1}^{N(j;c_{1})}\int_{U(\bar{\theta}_{(j,k)})}(1-\psi_{j,k})
\frac{\mathrm{d}P_{\theta,\varepsilon^{2}}}{\mathrm{d}P_{\theta_{0},\varepsilon^{2}}}
\mathrm{d}\mathrm{S}_{M}(\theta\mid\alpha) \bigg{]}  \nonumber\\
&\leq \exp\{(c_{2}+2)(B/\varepsilon)^{2 / (2\alpha_{0}+1)}\}
\sum_{j=C}^{\infty}N(j;c_{1})\exp\{-(j/32)(B/\varepsilon)^{2/(2\alpha_{0}+1)}\}
\nonumber\\
&\leq \exp\{(2c_{1} + c_{2} + 6 - C/32) (B/\varepsilon)^{2/(2\alpha_{0}+1)}\}.
\label{eq:bound_1}
\end{align}

Consider an upper bound on $T_{2}$ in (\ref{eq:decomp_three}).
Since $1-\exp\{- 1 / 8 \} > 1/\mathrm{e}^{3}$,
we have
\begin{align}
T_{2}&\leq \mathrm{E}_{\theta_{0},\varepsilon^{2}}
	\bigg{[}1_{H(\theta_{0})}(X)\sum_{j=C}^{\infty}\sum_{k=1}^{N(j;c_{1})}\psi_{j,k}(X)\bigg{]}
\nonumber\\
&\leq \sum_{j=C}^{\infty}\sum_{k=1}^{N(j;c_{1})}
\exp\{-(j/8)(B/\varepsilon)^{2/(2\alpha_{0}+1)} \}
\nonumber\\
&\leq  \exp\{(2c_{1} + 3 - C/8)(B/\varepsilon)^{2/(2\alpha_{0}+1)}\}.
\label{eq:bound_2}
\end{align}
Here, the second inequality follows from Lemma \ref{lem:testcondition},
and the third inequality follows from Lemma \ref{lem:coveringnumber}.

For an upper bound on $T_{3}$ in (\ref{eq:decomp_three}),
it follows that
\begin{align}
T_{3}&\leq \exp\{(c_{2}+2)(B/\varepsilon)^{2/(2\alpha_{0}+1)}\}
	\mathrm{E}_{\theta_{0},\varepsilon^{2}}
	\int_{E^{\mathrm{c}}_{c_{1}}(\theta_{0})}\frac{\mathrm{d}P_{\theta,\varepsilon^{2}} }{\mathrm{d}P_{\theta_{0},\varepsilon^{2}}}\mathrm{d}\mathrm{S}_{M}(\theta\mid\alpha)
\nonumber\\
&\leq \exp\{(c_{2}+2)(B/\varepsilon)^{2/(2\alpha_{0}+1)}\}
\mathrm{S}_{M}(E^{\mathrm{c}}_{c_{1}}(\theta_{0}) \mid \alpha)
\nonumber\\
&\leq
\exp\{(c_{2}+2+\eta -\eta c_{1})(B/\varepsilon)^{2/(2\alpha_{0}+1)}\}.
\label{eq:bound_3}
\end{align}
The first inequality follows from (\ref{eq:lowerbound_denom}).
The second inequality follows from Fubini's theorem.
The third inequality follows from Lemma \ref{lem:priormasscondition}.

Thus,
using (\ref{eq:decomp_firstterm}), (\ref{eq:bound_1}), (\ref{eq:bound_2}), and (\ref{eq:bound_3}) for an upper bound on (\ref{eq:decomp}),
and
taking $c_{1}$ and $C$ such that
\begin{align}
c_{2}+2+\eta- \eta c_{1}<0,\, 2c_{1} + 3 - C/8<0,\,\text{and}\,\, 2c_{1}+c_{2}+6-C/32<0,
\label{condition:C}
\end{align}
we complete the proof.

\qed

\subsection{Proof of Theorem \ref{thm:AdaptivePosteriorConcentrationwrtB}}
\label{subsec:subproof}

We provide the proof of Theorem \ref{thm:AdaptivePosteriorConcentrationwrtB}.
Replacing Lemmas \ref{lem:Essentialsupport} and \ref{lem:priormasscondition}
by Lemmas \ref{lem:AdaptiveEssentialsupport} and \ref{lem:Adaptivepriormasscondition}, respectively,
completes the proof,
because the other lemmas used in the proof of Theorem \ref{thm:PosteriorConcentrationwrtB}
do not depend on prior distributions.
The proofs of Lemmas \ref{lem:AdaptiveEssentialsupport} and \ref{lem:Adaptivepriormasscondition} are provided in Appendix \ref{subsec:proofs_lemmas}.

\begin{lem}\label{lem:AdaptiveEssentialsupport}
	For any $c_{1}>1$, the inequality
	\begin{align*}
		\Pi(E^{\mathrm{c}}_{c_{1}}(\theta_{0}) ) 
		\leq \exp\{ - \eta(c_{1}-1)(B/\varepsilon)^{2/(2\alpha_{0}+1)}\}
	\end{align*}
	holds uniformly in $\theta_{0}\in\mathcal{E}(\alpha_{0},B)$.
	Here, $\eta$ is a hyperparameter of $M$.
\end{lem}

\begin{lem}\label{lem:Adaptivepriormasscondition}
	There exists a positive constant $c_{2}$ depending only on $\alpha_{0}$, $\eta$ of $M$, and $\gamma$ of $F$
	for which the inequality
	\begin{align*}
		\Pi( \theta: \|\theta-\theta_{0}\|^2 /\varepsilon^2 \leq 2 (B/\varepsilon)^2 )
		\geq \exp\{-c_{2}(B/\varepsilon)^{2/(2\alpha_{0}+1)} \}
	\end{align*}
	holds uniformly in $\theta_{0}\in \mathcal{E}(\alpha_{0},B)$,
	provided that $\varepsilon/B$ is smaller than one.
\end{lem}

\subsection{Proof of Corollary \ref{cor:finite}}\label{subsec:proofofcor}

It suffices to derive a risk bound for the posterior mean $\hat{\theta}_{\Pi^{(p)}}$
in estimation of $(\theta_{1},\ldots,\theta_{p})$ based on i.i.d.~$p$ observations $X_{i}$ from $\mathcal{N}(\theta_{i},1/n)$.
The reason is as follows.
For $i\in\mathbb{N}$, let $\theta_{i}=\int f \phi_{i}\mathrm{d}t$.
We focus on $\sum_{i=1}^{p}(\theta_{i}-\hat{\theta}_{\Pi^{(p)},i})^2$ instead of $\|f-\hat{f}\|^2_{L_2}$
since
it follows
from the Parseval inequality
\begin{align*}
\|f-\hat{f}_{\Pi^{(p)}}\|_{L_2}^2 = \sum_{i=1}^{p}(\theta_{i}-\hat{\theta}_{\Pi^{(p)},i})^2
+ \sum_{i=p+1}^{\infty}\theta_{i}^{2}.
\end{align*}
The bias term appearing in the true distribution of $\{Y_{1},\ldots,Y_{n}\}$
due to $\{\phi_{j}(k/n): j=p+1,p+2,\ldots, k=1,\ldots,n\}$
is negligible
by a sufficiency reduction
because
the vectors 
$\{(\phi_{i}(1/n) , \ldots , \phi_{i}(1) )^{\top} / \sqrt{n} : i=1,\ldots, p\}$ 
are orthonormals in $\mathbb{R}^{p}$, provided that $p<n$.

Let $\underline{p}:= \min\{p, (nB^2)^{1/(2\alpha_{0}+1)}\}$.
To complete the proof,
it suffices to show that,
for a sufficiently large $C>0$ depending on $\alpha_{0}$, $\eta$, and $\gamma$,
the inequality
\begin{align*}
	\sup_{\theta^{(p)}_{0}: \sum_{i=1}^{p}i^{2\alpha_{0}}\theta^{(p),2}_{i,0} \leq B^2 }&\mathrm{E}_{\theta^{(p)}_{0},\varepsilon^{2}}
	{\Pi}^{(p)}
	 \bigg{(} n \sum_{i=1}^{p}\{\theta_{i}-\theta^{(p)}_{0,i}\}^{2} \geq
	C \underline{p} \bigg{)}
        \leq \exp\{-c \underline{p}\}
\end{align*}
holds,
where $c$ is a constant depending only on $\alpha_{0},\eta,$ and $\gamma$.
This is proved as follows.
If $p$ is larger than $c_{1}(n B^{2})^{1/(2\alpha_{0}+1)}$ for the constant $c_{1}$ appearing in the proof of Theorem \ref{thm:PosteriorConcentrationwrtB},
then the same proof as that of Theorem \ref{thm:AdaptivePosteriorConcentrationwrtB} is available.
Consider the case in which $p$ is smaller than $c_{1}(nB^2)^{1/(2\alpha_{0}+1)}$.
In this case, we replace $(n B^{2})^{1/(2\alpha_{0}+1)}$ by $p$.
This replacement does not change the conclusion as discussed below.
Lemmas \ref{lem:testcondition} and \ref{lem:highprobabilityset} do not change because their proofs rely only on the properties of the Gaussian measure.
Lemma \ref{lem:coveringnumber} still holds because the log of the covering number is bounded by $2p$.
Lemma \ref{lem:AdaptiveEssentialsupport} obviously holds for ${\Pi}^{(p)}$,
because $E_{c_{1}}(\theta_{0})$ is $\mathbb{R}^{p}$ itself for the case in which
$p < c_{1}(n B^{2} )^{1/(2\alpha_{0}+1)}$.
Lemma \ref{lem:Adaptivepriormasscondition} holds for $\mathrm{S}^{(p)}_{M,\alpha_{0}}$,
because the required lemma (Lemma \ref{lem:smallballestimate}) for the proof of Lemma \ref{lem:Adaptivepriormasscondition} is still available.
This completes the proof.
\qed

\section{Discussion}
\label{sec:discussion}


We propose two principal future studies.
The first study is to find a means of attaining non-asymptotic adaptation in the empirical Bayesian manner.
Recently, 
Petrone et al.~\cite{PetroneRousseauScricciolo(2014)} 
and Rousseau and Szab\'{o} \cite{RousseauandSzabo(2017)} established important asymptotic results
on the performance of empirical Bayesian nonparametrics.
Focusing on a simple setting, 
we expect to be able to answer whether there exists an empirical Bayesian method attaining non-asymptotic adaptation.
The investigation would also provide an insight into the relationship among
Bayesian nonparametrics, empirical Bayesian nonparametrics, model selection, and frequentist model averaging.
The second is to investigate non-asymptotic Bayesian adaptation in the other settings.
The present paper used a Gaussian infinite sequence model under Sobolev-type parameter constraints
for simplicity and for clarity of presentation.
One possible extension of our work would be to investigate
non-asymptotic Bayesian adaptation
in density estimation.
Resolution of these problems will increase our understanding of nonparametric estimation.

\appendix

\section{Proof for Section \ref{sec:Failures}}
\label{Appendix:Proof for Section Failures}

In this appendix, 
we provide the proof of Proposition \ref{prop:model selection and model averaging} 
in the case with the model averaging estimator.
To apply the argument in Section 7.B.~of \cite{LeungandBarron(2006)},
slight modifications to the loss functions and the weights are necessary,
because Leung and Barron \cite{LeungandBarron(2006)} use the finite dimensional $l_2$ loss function 
and
assume that the number of models is finite.
Although these modifications are straightforward, we provide them for the sake of completeness.

\begin{proof}
Let $D:=\lfloor (B/\varepsilon)^{2/(2\alpha_{0}+1)} \rfloor$.
The $l_{2}$ risk $\mathrm{E}_{\theta,\varepsilon^2}\|\theta-\hat{\theta}_{\mathrm{MA},\beta}\|^2$ 
is decomposed as follows:
	\begin{align}
		\mathrm{E}_{\theta,\varepsilon^2}\|\theta-\hat{\theta}_{\mathrm{MA},\beta}\|^2
		= \mathrm{E}_{\theta,\varepsilon^{2}} \sum_{i=1}^{D} (\theta_{i} - \hat{\theta}_{\mathrm{MA},\beta, i} )^{2}
		+ \mathrm{E}_{\theta,\varepsilon^{2}} \sum_{i=D+1}^{\infty} (\theta_{i} - \hat{\theta}_{\mathrm{MA},\beta,i})^{2}.
		\label{eq: decomp MA risk}
\end{align}

First, we show that the latter term on the right hand side in (\ref{eq: decomp MA risk}) is $\mathrm{O}(D)$.
The latter term on the right hand side in (\ref{eq: decomp MA risk}) is bounded above as
\begin{align*}
		\mathrm{E}_{\theta,\varepsilon^{2}} \sum_{i=D+1}^{\infty} (\theta_{i} - \hat{\theta}_{\mathrm{MA},\beta,i})^{2}
		& \leq 2\sum_{i=D+1}^{\infty}\theta_{i}^{2}
		+ 2 \mathrm{E}_{\theta,\varepsilon^{2}} \sum_{i=D+1}^{\infty}\hat{\theta}_{\mathrm{MA,\beta,i}}^{2}
		\\
		& \leq 6 \sum_{i=D+1}^{\infty} \theta_{i}^2 
		+ 4 \sqrt{3} \varepsilon^{2} \sum_{i=D+1}^{\infty} \bigg{\{} \mathrm{E}_{\theta,\varepsilon^{2}} \bigg{(} \sum_{d=i}^{\infty} w_{d} \bigg{)}^4 \bigg{\}}^{1/2},
\end{align*}
where the second inequality follows from the fact that $\hat{\theta}_{\mathrm{MA} , \beta , i}= \sum_{d=1}^{\infty} w_{d}X_{i}1_{i \leq d} $
and
from the Cauchy--Schwarz inequality.
Consider an upper bound on the latter term in the above inequality.
Note that the inequality
\begin{align*}
	\sum_{d=i}^{\infty}w_{d} \leq \sum_{d=i}^{\infty}\exp\bigg{[} \{\beta / (2\varepsilon^{2})\} \bigg{\{}\sum_{j=D}^{d}X_{j}^{2}-2\varepsilon^2 (d-D) \bigg{\}} \bigg{]}
\end{align*}
holds.
For $i\geq 25D$ and for some $s>0$,
\begin{align*}
	&\sum_{d=i}^{\infty}\mathrm{Pr}\bigg{(}\sum_{j=D}^{d}X_{j}^{2} -2\varepsilon^2 (d-D) > -\varepsilon^2 (d-D) / 4 \bigg{)} \\
	&\leq \sum_{d=i}^{\infty}\mathrm{Pr}\bigg{(} 3\sum_{j=D}^{d}N_{j}^2/2 + 3 \sum_{j=D}^{\infty}\theta_{j}^2/\varepsilon^2
	> 7(d-D)/4   \bigg{)} \\
	&\leq \exp\{-s(i-D)\},
\end{align*}
where the first inequality follows because $xy\leq x^2/2+ 2y^2$ for $x,y>0$,
and
the second inequality follows because $\sum_{j=D}\theta_{j}^2/\varepsilon^2  \leq D^{2\alpha_{0}+1}D^{-2\alpha_{0}} $,
because $d> 25D$,
and
from the Borell--Sudakov--Tsirelson Gaussian concentration inequality.
Therefore, there exists a universal positive constant $s'$ for which
\begin{align*}
	\mathrm{E}_{\theta,\varepsilon^{2}} \sum_{i=D+1}^{\infty} (\theta_{i} - \hat{\theta}_{\mathrm{MA},\beta,i})^{2}
	\leq \varepsilon^{2} \{6D+100\sqrt{3}D+4\sqrt{3}\exp(-s'D)\}.
\end{align*}

Second, we show that the former term on the right hand side in (\ref{eq: decomp MA risk}) is bounded as
\begin{align*}
\mathrm{E}_{\theta,\varepsilon^{2}} \sum_{i=1}^{D} (\theta_{i} - \hat{\theta}_{\mathrm{MA},\beta, i} )^{2}
\leq \mathrm{E}_{\theta,\varepsilon^{2}} \sum_{d=1}^{D}\tilde{w}_{d}^{(D)}\hat{r}_{d}^{(D)} + D\varepsilon^{2},
\end{align*}
where
for $d=1,\ldots,D$,
$\tilde{w}^{(D)}_{d} := w_{d} / \sum_{d'=1}^{D} w_{d'}$ 
and
for $d=1,2,\ldots$,
$\hat{r}_{d}^{(D)} := \sum_{i=1}^{D}(X_{i}-X_{i}1_{i\leq d})^2 - D\varepsilon^2 +2\min\{D,d\} \varepsilon^2 $.
Let 
\begin{align*}
	\hat{r}^{(D)}:=\sum_{d=1}^{\infty}w_{d}\bigg{[}\hat{r}^{(D)}_{d}-(1-2\beta)\sum_{i=1}^{D}(X_{i}1_{i\leq d}-\hat{\theta}_{\mathrm{MA},\beta,i} )^2 \bigg{]}.
\end{align*}
Noting that $\hat{r}^{(D)}_{d}$ is a risk unbiased estimator of $\mathrm{E}_{\theta,\varepsilon^2}[\sum_{i=1}^{D}(X_{i}1_{i\leq d}-\theta_{i})^2]$
and
that $\hat{r}^{(D)}$ is a risk unbiased estimator of $\mathrm{E}_{\theta,\varepsilon^2} [\sum_{i=1}^{D}( \hat{\theta}_{\mathrm{MA},\beta,i} - \theta_{i} )^2 ]$,
we have
\begin{align*}
	\mathrm{E}_{\theta,\varepsilon^2}\sum_{i=1}^{D}(\theta_{i}-\hat{\theta}_{\mathrm{MA},\beta,i})^2
	=\mathrm{E}_{\theta,\varepsilon^2}\hat{r}^{(D)}
	\leq \mathrm{E}_{\theta,\varepsilon^2} \sum_{d=1}^{\infty}w_{d}\hat{r}^{(D)}_{d} ,
\end{align*}
where we use the condition that $\beta\leq 1/2$.
Thus,
we obtain
\begin{align*}
\mathrm{E}_{\theta,\varepsilon^2} \sum_{i=1}^{D}(\theta_{i}-\hat{\theta}_{\mathrm{MA},\beta,i})^2
\leq 
\mathrm{E}_{\theta,\varepsilon^2} \sum_{d=1}^{D}\tilde{w}_{d}^{(D)}\hat{r}^{(D)}_{d} + D\varepsilon^{2}.
\end{align*}

Finally,
since
\begin{align*}
	\sum_{d=1}^{D}\tilde{w}^{(D)}_{d}\hat{r}^{(D)}_{d}
	\leq \min_{d=1, \ldots , D}\hat{r}^{(D)}_{d} + (2\varepsilon^{2} /\beta) \log D,
\end{align*}
applying the argument in Section 7.B.~in \cite{LeungandBarron(2006)} completes the proof.

\end{proof}

\section{Proofs of lemmas in Section \ref{sec:details_ss}}
\label{subsec:proofs_lemmas}

In this appendix,
we provide the proofs of Lemmas \ref{lem:Essentialsupport}, \ref{lem:priormasscondition}, \ref{lem:highprobabilityset}, \ref{lem:AdaptiveEssentialsupport}, \ref{lem:Adaptivepriormasscondition}.
For the proof of Lemma \ref{lem:coveringnumber}, see Proposition A.1 in \cite{GaoandZhou(2016)}.
For the proof of Lemma \ref{lem:testcondition}, see Lemma 5 in \cite{GhosalandvanderVaart(2007)}.

\begin{proof}[Proof of Lemma \ref{lem:Essentialsupport}]
Let $\bar{D}=\lfloor c_{1}(B/\varepsilon)^{2/(2\alpha_{0}+1)}\rfloor$.
From the definition of $\mathrm{S}_{M}(\cdot\mid\alpha )$,
\begin{align*}
\mathrm{S}_{M} (E_{c_{1}}^{\mathrm{c}}(\theta_{0}) \mid \alpha)
&=\sum_{d=1}^{\bar{D} }M(d) \mathrm{S} \bigg{(} \sum_{i> \bar{D} }
 (\theta_{i}-\theta_{0,i})^{2}/ \varepsilon^{2}
 > (B / \varepsilon )^{2/(2\alpha_{0}+1)}  \mid d, \alpha \bigg{)}
\nonumber\\
&\quad +\sum_{d=\bar{D} +1}^{\infty}M(d) \mathrm{S}
\bigg{(}\sum_{i>\bar{D} }
(\theta_{i}-\theta_{0,i})^{2} / \varepsilon^{2}
> ( B / \varepsilon^{2} )^{ 1 / (2\alpha_{0}+1)} \mid d , \alpha \bigg{)}.
\end{align*}
The first term on the right hand side of the above equality vanishes, because
the identity
\begin{align*}
	\mathrm{S} \bigg{(}\sum_{i>\bar{D}}
 (\theta_{i}-\theta_{0,i})^{2} / \varepsilon^{2}
 > ( B / \varepsilon )^{2/(2\alpha_{0}+1)} \mid d, \alpha\bigg{)}
=0,
\end{align*}
holds for $d\leq \bar{D}$
since
$\sum_{i>\bar{D} }\theta_{0,i}^{2} / \varepsilon^{2}
\leq (1/ c_{1}^{2\alpha_{0}}) ( B / \varepsilon)^{2/(2\alpha_{0}+1)}
< (B / \varepsilon )^{2/(2\alpha_{0}+1)}.$
The second term is bounded by
$\exp\{\eta - \eta c_{1}(B/\varepsilon)^{2/(2\alpha_{0}+1)}\}.$
This completes the proof.
\end{proof}

\begin{proof}[Proof of Lemma \ref{lem:priormasscondition}]
The proof relies on the following lemma.
Let $\{N_{i}\}_{i=1}^{\infty}$ be independent random series from the standard Gaussian distribution.
\begin{lem}\label{lem:smallballestimate}
	For each $\alpha>0$,
	there exists a positive constant $c_{3}$ depending only on $\alpha$
	such that,
	for a sufficiently large $d\in\mathbb{N}$
	and
	for each $(v_{1},\ldots,v_{d})\in \mathbb{R}^{d}$,
	the inequality
	\begin{align*}
		\mathrm{Pr} \bigg{(}\sum_{i=1}^{d}(i^{-\alpha-1/2}N_{i}-v_{i})^{2}\leq d^{-2\alpha} \bigg{)}
		\geq 
		\exp\bigg{\{}- \sum_{i=1}^{d}i^{2\alpha +1}v_{i}^{2} / 2 - c_{3} d \bigg{\}}
	\end{align*}
	holds.
\end{lem}
The proof of the lemma is provided in Appendix \ref{Appendix:smallballestimate} for the sake of completeness.

Return to the proof of Lemma \ref{lem:priormasscondition}.
Let $T= \lfloor (B/\varepsilon)^{2/(2\alpha_{0}+1)} \rfloor$.
From Lemma \ref{lem:smallballestimate} with $d=T$, $v_{i}= \theta_{0,i}/ (\varepsilon T^{\alpha+1/2})$, $i=1,\ldots,d$,
\begin{align}
\mathrm{S} ( \|\theta-\theta_{0}\|^{2} / \varepsilon^{2} \leq 2T \mid  \alpha , d=\lfloor T\rfloor )
&\geq
\mathrm{Pr} \bigg{(} \sum_{i=1}^{T}(\varepsilon T^{\alpha+1/2} i^{-\alpha-1/2}N_{i}- \theta_{0,i})^2 / \varepsilon^{2}
\leq T \bigg{)}
\nonumber\\
&\geq \exp\bigg{\{}-\sum_{i=1}^{T}i^{2\alpha+1}\theta_{0,i}^{2}/(\varepsilon^{2}T^{2\alpha+1}) - c_{3} T \bigg{\}}
\nonumber\\
&\geq \exp\{- (c_{3}+1)T\}.
\label{eq:lowerbound_S_component}
\end{align}
Here,
the first inequality follows because
$\sum_{i=T+1}^{\infty}\theta_{0,i}^{2}/\varepsilon^{2} \leq T$.
The second inequality follows from Lemma \ref{lem:smallballestimate}.
The last inequality follows because
\[\sum_{i=1}^{T}i^{2\alpha+1}\theta_{0,i}^{2} \big{/} \varepsilon^{2} 
= \sum_{i=1}^{T}i^{2(\alpha-\alpha_{0})+1} i^{2\alpha_{0}}\theta_{0,i}^{2} \big{/} \varepsilon^{2} \leq T^{2\alpha+2},\]
where we use the condition that $2(\alpha-\alpha_{0})+1>0$.

Using the inequality (\ref{eq:lowerbound_S_component}) and the inequality $1/(1-\mathrm{e}^{-\eta}) \geq \mathrm{e}^{-\eta}$ for $\eta>0$
yields
\begin{align*}
\mathrm{S}_{M}( \|\theta-\theta_{0}\|^{2} / \varepsilon^{2} \leq 2T, \alpha)
&\geq M(T)\mathrm{S}(\|\theta-\theta_{0}\|^{2} / \varepsilon^{2} \leq 2T \mid d=  T, \alpha)
\nonumber\\
&\geq \exp\{-(c_{3}+2+\eta )T\}.
\end{align*}
This completes the proof.
\end{proof}

\begin{proof}[Proof of Lemma \ref{lem:highprobabilityset}]
By Jensen's inequality,
\begin{align*}
	P_{\theta_{0},\varepsilon^{2}}[H(\theta_{0})]
	&=P_{\theta_{0},\varepsilon^{2}}\bigg{[}\log \int \frac{\mathrm{d}P_{\theta,\varepsilon^{2}} }{\mathrm{d}P_{\theta_{0},\varepsilon^{2}}} \mathrm{d}\widetilde{\mathrm{S}}_{M}(\theta\mid\alpha)
\geq -2(B/\varepsilon)^{2/(2\alpha_{0}+1)} \bigg{]}
	\nonumber\\
	&\geq
	P_{\theta_{0},\varepsilon^{2}}\bigg{[} \int \log \frac{\mathrm{d}P_{\theta,\varepsilon^{2}}}{ \mathrm{d}P_{\theta_{0},\varepsilon^{2}}} \mathrm{d}\widetilde{\mathrm{S}}_{M}(\theta\mid\alpha)
	\geq -2(B/\varepsilon)^{2/(2\alpha_{0}+1)} \bigg{]}.
\end{align*}
Since $\widetilde{\mathrm{S}}_{M}(\cdot\mid\alpha)$ has a support on 
$\{\theta:\|\theta-\theta_{0}\|^{2}/\varepsilon^{2} \leq 2(B/\varepsilon)^{2/(2\alpha_{0}+1)}\}$,
we have
\begin{align*}
	\int \log \frac{\mathrm{d}P_{\theta,\varepsilon^{2}}}{\mathrm{d}P_{\theta_{0},\varepsilon^{2}}}
	\mathrm{d}\widetilde{\mathrm{S}}_{M}(\theta\mid\alpha)
	&=
	\int \bigg{\{} \frac{\langle (X-\theta_{0}),(\theta-\theta_{0})\rangle}{\varepsilon^{2}}-\frac{||\theta-\theta_{0}||^{2}}{2\varepsilon^{2}} \bigg{\}}
	\mathrm{d}\widetilde{\mathrm{S}}_{M}(\theta \mid \alpha)
	\nonumber\\
	&\geq
	\int \bigg{\{} \frac{\langle (X-\theta_{0}), (\theta-\theta_{0}) \rangle }{\varepsilon^{2}} \bigg{\}}
	\mathrm{d}\widetilde{\mathrm{S}}_{M}(\theta \mid \alpha)
	- (B/\varepsilon)^{2/(2\alpha_{0}+1)},
\end{align*}
where $\langle \cdot,\cdot \rangle$ is the $l_{2}$-inner product. 
Thus, letting $N$ be a one-dimensional standard normal random variable yields
\begin{align*}
	P_{\theta_{0},\varepsilon^{2}}[H^{\mathrm{c}}(\theta_{0})]
	&\leq
	P_{\theta_{0},\varepsilon^{2}}\bigg{[}\int \frac{\langle (X-\theta_{0}) , (\theta-\theta_{0}) \rangle}{\varepsilon^{2}} \mathrm{d}\widetilde{\mathrm{S}}_{M}(\theta \mid \alpha)
	< -(B/\varepsilon)^{2/(2\alpha_{0}+1)} \bigg{]}
	\nonumber\\
	&=
	\mathrm{Pr}\bigg{[}\int \sqrt{\left(\frac{||\theta_{0}-\theta ||^{2}}{\varepsilon^{2}}\right)}N \mathrm{d}\widetilde{\mathrm{S}}_{M}(\theta \mid \alpha)
	>(B/\varepsilon)^{2/(2\alpha_{0}+1)} \bigg{]}
	\nonumber\\
	&\leq
	\mathrm{Pr}[\sqrt{(B/\varepsilon)^{2/(2\alpha_{0}+1)}} N >(B/\varepsilon)^{2/(2\alpha_{0}+1)}]
	\nonumber\\
	&\leq \exp\{-(1/2)(B/\varepsilon)^{2/(2\alpha_{0}+1)}\}.
\end{align*}
Here,
for the last inequality,
we use the inequality $\mathrm{Pr}(|N|>r)\leq \exp(-r^{2}/2), r>0$.
\end{proof}

\begin{proof}[Proof of Lemma \ref{lem:AdaptiveEssentialsupport}]
	This follows immediately
	from the identity $\Pi=\sum_{k=1}^{\infty}F(k) \allowbreak\mathrm{S}_{M}(\cdot\mid \alpha=k)$
	and from Lemma \ref{lem:Essentialsupport}.
\end{proof}

\begin{proof}[Proof of Lemma \ref{lem:Adaptivepriormasscondition}]
	Take an integer $\overline{k} \geq \alpha_{0}-1/2$
	and take $T=\lfloor (B/\varepsilon)^{2/(2\alpha_{0}+1)} \rfloor$.
	Then,
	\begin{align*}
	\Pi(\|\theta-\theta_{0}\|^2 / \varepsilon^2 \leq 2T)
	&\geq F(\overline{k})\mathrm{S}_{M}( \|\theta-\theta_{0}\|^2 / \varepsilon^2 \leq 2T \mid \alpha= \overline{k} ) \\
	&\geq F(\overline{k})M(T)\mathrm{S}( \|\theta-\theta_{0}\|^2 / \varepsilon^2 \leq 2T \mid \alpha=\overline{k} , d=T)\\
	&\geq \exp\{ -1 -\gamma \overline{k} - (c_{3}+2+\eta) T \},
	\end{align*}
	where $c_{3}$ is a positive constant 
	depending only on $\overline{k}$ and 
	appearing in Lemma \ref{lem:smallballestimate}
	and the last inequality follows from the proof of Lemma \ref{lem:priormasscondition}.
	This completes the proof.
\end{proof}

\section{Proof of Lemma \ref{lem:smallballestimate}}
\label{Appendix:smallballestimate}

In this appendix, 
for the sake of completeness,
we provide the proof of 
an important inequality for estimating the small ball probability
(Lemma \ref{lem:smallballestimate}).

\begin{proof}
	Since the distributions of $N_{1}$ and $-N_{1}$ are identical,
	we have
	\begin{align*}
		\mathrm{Pr}& \bigg{(}\sum_{i=1}^{d}(i^{-\alpha-1/2}N_{i}-v_{i})^2 \leq d^{-2\alpha} \bigg{)}
		\\
		=&\mathrm{Pr} \bigg{(}\sum_{i=1}^{d}(i^{-\alpha-1/2}N_{i}-v_{i})^2 \leq d^{-2\alpha} \bigg{)}/2
		+\mathrm{Pr} \bigg{(}\sum_{i=1}^{d}(i^{-\alpha-1/2}N_{i}+v_{i})^2 \leq d^{-2\alpha} \bigg{)}/2
		\\
		=&\int_{\sum_{i=1}^{d}i^{-2\alpha-1}x_{i}^2 \leq d^{-2\alpha} }
		\mathrm{cosh}\bigg{(}\sum_{i=1}^{d}i^{\alpha+1/2}x_{i}v_{i}\bigg{)}
		\frac{\exp\{-\sum_{i=1}^{d}x_{i}^{2}/2-\sum_{i=1}^{d}i^{2\alpha+1}v_{i}^{2}/2 \} }{(2\pi)^{d/2}}\mathrm{d}x
		\\
		\geq &\exp\bigg{\{}-\sum_{i=1}^{d}i^{2\alpha+1}v_{i}^{2}/2 \bigg{\}}
		\int_{\sum_{i=1}^{d}i^{-2\alpha-1}x_{i}^2 \leq d^{-2\alpha} } \frac{ \exp\{-\sum_{i=1}^{d}x_{i}^2 /2 \} }{(2\pi)^{d/2}}\mathrm{d}x.
	\end{align*}

	Second, 
	we show that 
	there exists a positive constant $c_{3}$ depending only on $\alpha$ such that
	for $d\in\mathbb{N}$,
	the inequality
	\begin{align*}
	\int_{\sum_{i=1}^{d}i^{-2\alpha-1}x_{i}^2 \leq d^{-2\alpha} } \frac{ \exp\{-\sum_{i=1}^{d}x_{i}^2 /2 \} }{(2\pi)^{d/2}}\mathrm{d}x
	\geq \exp\{-c_{3}d\}
	\end{align*}
	holds.
	Changing variables yields
	\begin{align*}
		&\int_{\sum_{i=1}^{d}i^{-2\alpha-1}x_{i}^2 \leq d^{-2\alpha} } \frac{ \exp\{-\sum_{i=1}^{d}x_{i}^2 /2 \} }{ (2\pi)^{d/2} } \mathrm{d}x
		\\
		&\geq\{\Gamma(d+1)\}^{\alpha+1/2}\int_{\sum_{i=1}^{d}y_{i}^{2}\leq d^{-2\alpha} }
		\frac{\exp\{-d^{2\alpha+1}\sum_{i=1}^{d}y_{i}^{2}/2 \}}{(2\pi)^{d/2}}\mathrm{d}y
		\\
		&\geq \{\Gamma(d+1)\}^{\alpha+1/2}\frac{\exp(-d/2)}{(2\pi)^{d/2}}\int_{\sum_{i=1}^{d}y_{i}^{2}\leq d^{-2\alpha} }\mathrm{d}y.
	\end{align*}
	Since $\int_{\sum_{i=1}^{d}y_{i}^2 \leq d^{-2\alpha} }\mathrm{d}y = d^{-d\alpha} \pi^{d/2} / \Gamma(d/2+1) $,
	we have,
	for some universal constant $\tilde{c_{1}}$,
	\begin{align*}
&\int_{\sum_{i=1}^{d}i^{-2\alpha-1}x_{i}^2 \leq d^{-2\alpha} } \frac{ \exp\{-\sum_{i=1}^{d}x_{i}^2 /2 \} }{(2\pi)^{d/2}}\mathrm{d}x
\\
&\geq [ \{\Gamma(d+1)\}^{\alpha+1/2} / \{d^{d\alpha}\Gamma(d/2+1) \} ]\exp(-\tilde{c}_{1}d).
	\end{align*}
	Here,
	it follows from Stirling's formula that
	there exist positive constants $\tilde{c}_{2}$ and $\tilde{c}_{3}$
	depending only on $\alpha$ such that for $d\in\mathbb{N}$ the inequalities
	\begin{align*}
		&\{\Gamma(d+1)\}^{\alpha+1/2}\geq \exp\{(\alpha+1/2)(d+1/2)\log d - \tilde{c}_{2}d \}, \\
	        &d^{d\alpha}\Gamma(d/2+1)\leq \exp\{(\alpha+1/2) d\log d + \tilde{c}_{3}d\}
	\end{align*}
	hold,
	and thus we obtain
	\begin{align*}
		\int_{\sum_{i=1}^{d}i^{-2\alpha-1}x_{i}^{2} \leq d^{-2\alpha} }
		\frac{\exp\{ - \sum_{i=1}^{d}x_{i}^{2} / 2\}}{(2\pi)^{d/2}}\mathrm{d}x
		\geq
		\exp\{-(\tilde{c}_{1}+\tilde{c}_{2}+\tilde{c}_{3})d\}.
	\end{align*}
\end{proof}

\section{The explicit form of the posterior}
\label{Appendix: explicit form}
In this appendix, we provide an explicit form of the posterior of $\Pi$.
The explicit form of the posterior is useful when conducting numerical experiments.
The posterior of $\Pi$ is given by
\begin{align*}
	\Pi(\cdot\mid x) = \sum_{k=1}^{\infty}F( k \mid x) \sum_{d=1}^{\infty} M(d\mid x,k) \mathrm{S}(\cdot\mid x, d, k),
\end{align*}
where $\mathrm{S}(\cdot\mid x,d,k)$ is given by
\begin{align*}
	\bigg{[}\mathop{\otimes}_{i=1}^{d}\mathcal{N}\bigg{(}\bigg{(}1-\frac{1}{(d/i)^{2k+1}+1}\bigg{)}x_{i},
		\varepsilon^{2}\bigg{(}1-\frac{1}{(d/i)^{2k+1}+1}\bigg{)}\bigg{)}\bigg{]}
	\otimes
	\bigg{[}\mathop{\otimes}_{i=d+1}^{\infty}\mathcal{N}(0,0)\bigg{]},
\end{align*}
$M(\cdot\mid x,k)$ is given by
\begin{align*}
	M(d\mid x,k)\propto M(d)\prod_{i=1}^{d}\bigg{\{}1+\bigg{(}\frac{d}{i}\bigg{)}^{2k+1}\bigg{\}}^{-1/2}
	\exp\bigg{(} \sum_{i=1}^{d}\frac{x^{2}_{i}}
	{2\varepsilon^{2}}\frac{(d/i)^{2k+1}}{(d/i)^{2k+1}+1} \bigg{)},
\end{align*}
$F(\cdot\mid x)$ is given by
\begin{align*}
	F(k\mid x)\propto F(k)\sum_{d=1}^{\infty}M(d)\prod_{i=1}^{d}\bigg{\{} 1+ \bigg{(}\frac{d}{i}\bigg{)}^{2k+1} \bigg{\}}^{-1/2}
	\exp\bigg{(} \sum_{i=1}^{d}\frac{x^{2}_{i}}
	{2\varepsilon^{2}}\frac{(d/i)^{2k+1}}{(d/i)^{2k+1}+1}\bigg{)}.
\end{align*}
Here we omit the normalizing constant.

The derivation of the posterior form is as follows.
Letting $\mu^{(d)}$ be the product of the $d$-dimensional Lebesgue measure and $\otimes_{d+1}^{\infty}\mathcal{N}(0,0)$ together with the Bayes theorem
yields
\begin{align}
	\frac{\mathrm{d}\mathrm{S}(\cdot\mid x,d,k)}{\mathrm{d}\mu^{(d)}}
	&\propto\frac{\mathrm{d}\mathrm{S}(\cdot\mid d,k)}{\mathrm{d}\mu^{(d)}}(\theta^{(d)}) 
	\frac{\mathrm{d}P_{\theta,\varepsilon^{2}}}{\mathrm{d}P_{0,\varepsilon^{2}}}(x) \nonumber\\
	&\propto \exp\bigg{\{}-\sum_{i=1}^{d}\frac{\theta^{2}_{i}}{2\varepsilon^{2}(d/i)^{2k+1}}-\sum_{i=1}^{d}\frac{(x_{i}-\theta_{i})^{2} }{2\varepsilon^{2}} \bigg{\}}\nonumber\\
	&\propto \exp\bigg{\{}-\sum_{i=1}^{d}\frac{1}{2\varepsilon^{2}}\bigg{(}\frac{1}{(d/i)^{2k+1}}+1\bigg{)}
\bigg{\{}\theta_{i}-\frac{x_{i}}{1+1/(d/i)^{2k+1}}\bigg{\}}^{2} \bigg{\}}.
\end{align}
Thus,
we obtain the explicit form of $\mathrm{S}(\cdot\mid x, d,k)$.
Since the marginal distribution $P_{\mathrm{S}(\cdot\mid d,k)}$ of $x$ 
with respect to $\mathrm{S}(\cdot\mid d,k)$ is
$[\otimes_{i=1}^{d}\mathcal{N}(0,\varepsilon^{2}(1+(d/i)^{2k+1}))]
\otimes[\otimes_{d+1}^{\infty}\mathcal{N}(0,\varepsilon^{2})]$,
we obtain
\begin{align*}
	M(d\mid x,k)
	&\propto M(d)\frac{\mathrm{d}P_{\mathrm{S}(\cdot\mid d,k) }}{\mathrm{d}P_{0,\varepsilon^{2}}}(x) \nonumber\\
 &\propto M(d)\prod_{i=1}^{d} (1+(d/i)^{2k+1})^{-1/2}
	\exp\bigg{\{}-\frac{x^{2}_{i}}{2\varepsilon^{2}(1+(d/i)^{2k+1})}+\frac{x^{2}_{i}}{2\varepsilon^{2}}\bigg{\}}.
\end{align*}
A similar calculation yields the explicit form of $F(k\mid x)$.

\section{Supplementary numerical experiments}
\label{Appendix:secondary experiments}

In this appendix, we provide several numerical experiments aimed at assisting the readers' understanding.
The experimental setting is almost the same as that of Section \ref{sec:numerical}:
Recall that
numerical experiments are conducted with $p=100$-dimensional settings,
and that the noise variance $\varepsilon^2$ is fixed to one.
In addition to the estimators and the parameter values in Section \ref{sec:numerical},
we use the following estimators and parameter values:
\begin{itemize}
	\item the maximum likelihood estimator $(X_{1},\ldots,X_{p})$ ;
	\item the blockwise James--Stein estimator of which the truncation dimension is $p$;
	\item the Bayes estimator based on the Gaussian scale mixture prior distribution
              \[\int \otimes_{i=1}^{\infty}\mathcal{N}(0,ti^{-5}) \mathrm{d} V(t)\]
              with the discretized inverse Gamma distribution $V$ of which the rate and shape parameters are both one.
\end{itemize}
For $i=1,2,\ldots,$
\begin{itemize}
	\item $\hat{\theta}^{(3)}_{i}:=Bi^{-0.65}/\sqrt{4}$;
	\item $\hat{\theta}^{(4)}_{i}:=Bi^{-3}/\sqrt{\pi^{4}/90}$.
\end{itemize}

\subsection{Comparison between estimators with and without non-asymptotic adaptation}

We compare the performance between 
estimators with and without non-asymptotic adaptation
using white noise representation.
We represented
the true parameter $\theta$ as $t\in[0,1]\to\sum_{i=1}^{p}\theta_{i}\phi_{i}(t)$,
the observation $x$ as $t\in[0,1]\to\sum_{i=1}^{p}x_{i}\phi_{i}(t)$,
and an estimator $\hat{\theta}$ as $t\in[0,1]\to\sum_{i=1}^{p}\hat{\theta}_{i}(x)\phi_{i}(t)$,
where 
$\{\phi_{i}(\cdot)\}_{i=1}^{\infty}$ is the trigonometric series.
In Figures \ref{Nonscaleratio_1}--\ref{Scaleratio_2},
these are plotted at $\{0.001\times i\}_{i=1}^{1000}$.

\begin{figure}[h]
	\begin{minipage}{0.4\hsize}
	\begin{center}
	\includegraphics[width=0.90\hsize]{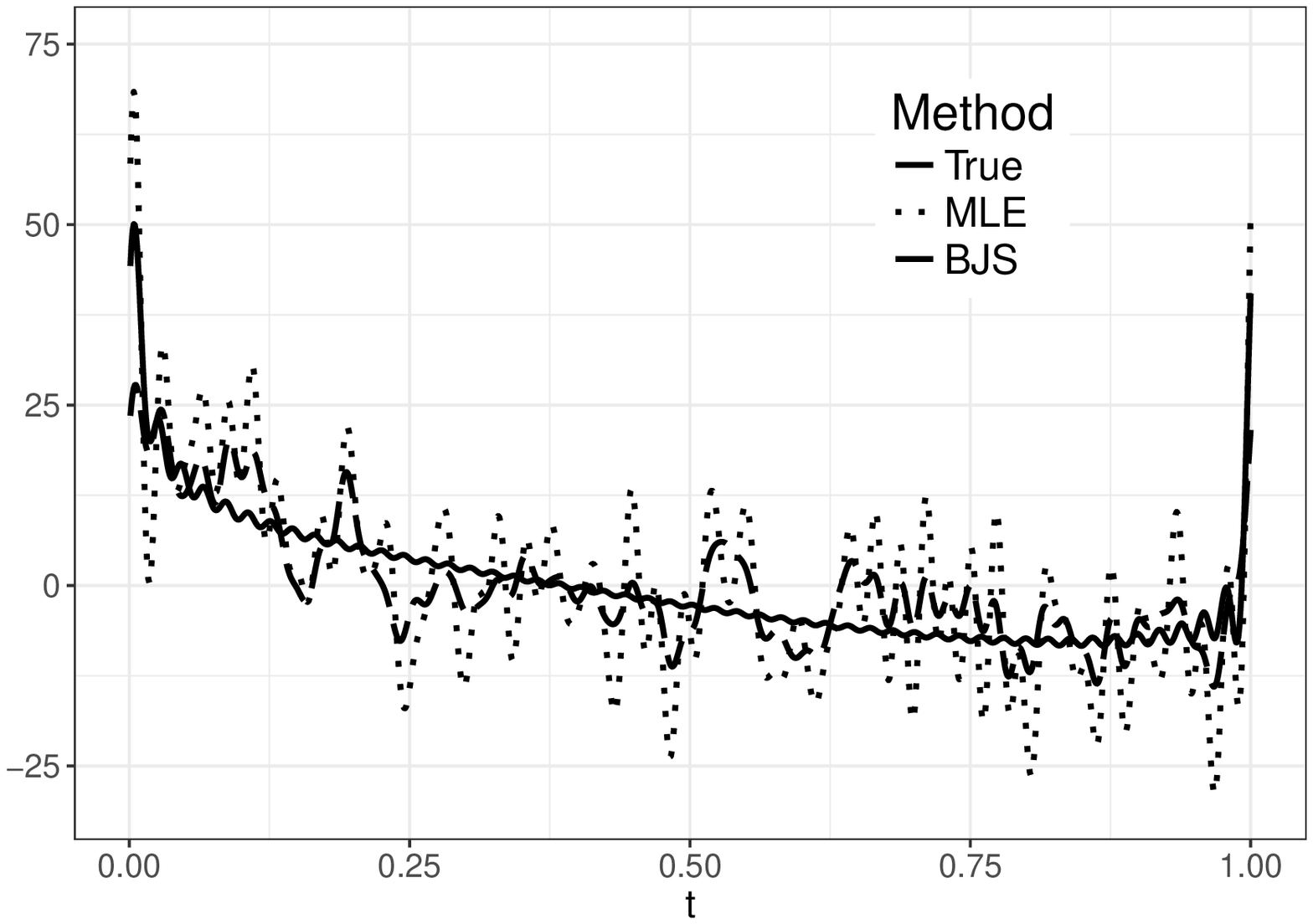}
	\caption{White noise representation of the true parameter and estimators without non-asymptotic adaptation at $\theta=\theta^{(3)}$ and $B=10$.}
	\label{Nonscaleratio_1}
	\end{center}
	\end{minipage}
	\begin{minipage}{0.4\hsize}
	\begin{center}
	\includegraphics[width=0.90\hsize]{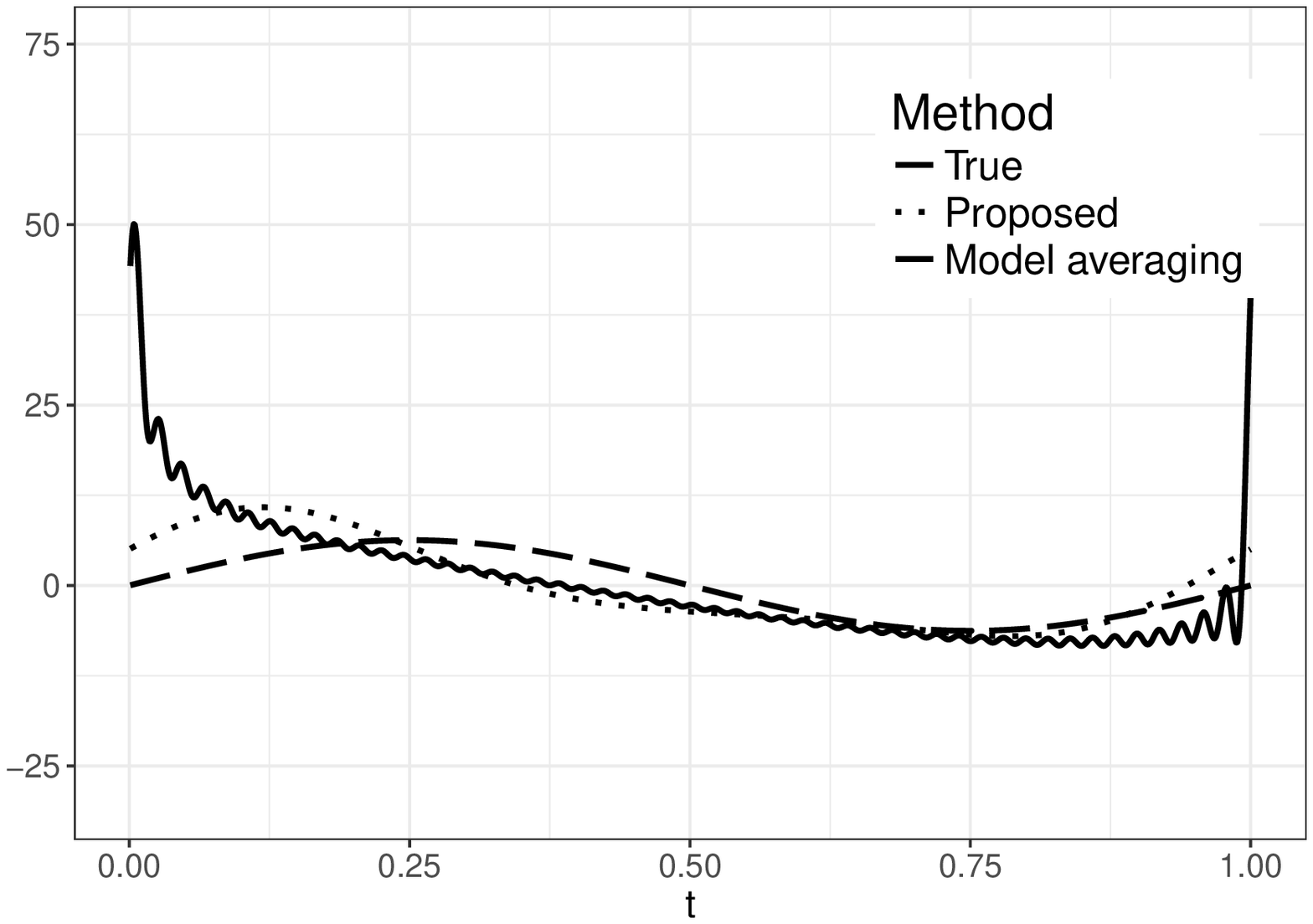}
	\caption{White noise representation of the true parameter and estimators with non-asymptotic adaptation at $\theta=\theta^{(3)}$ and $B=10$.}
	\label{Scaleratio_1}
	\end{center}
	\end{minipage}
\end{figure}

\begin{figure}[h]
	\begin{minipage}{0.4\hsize}
	\begin{center}
		\includegraphics[width=0.90\hsize]{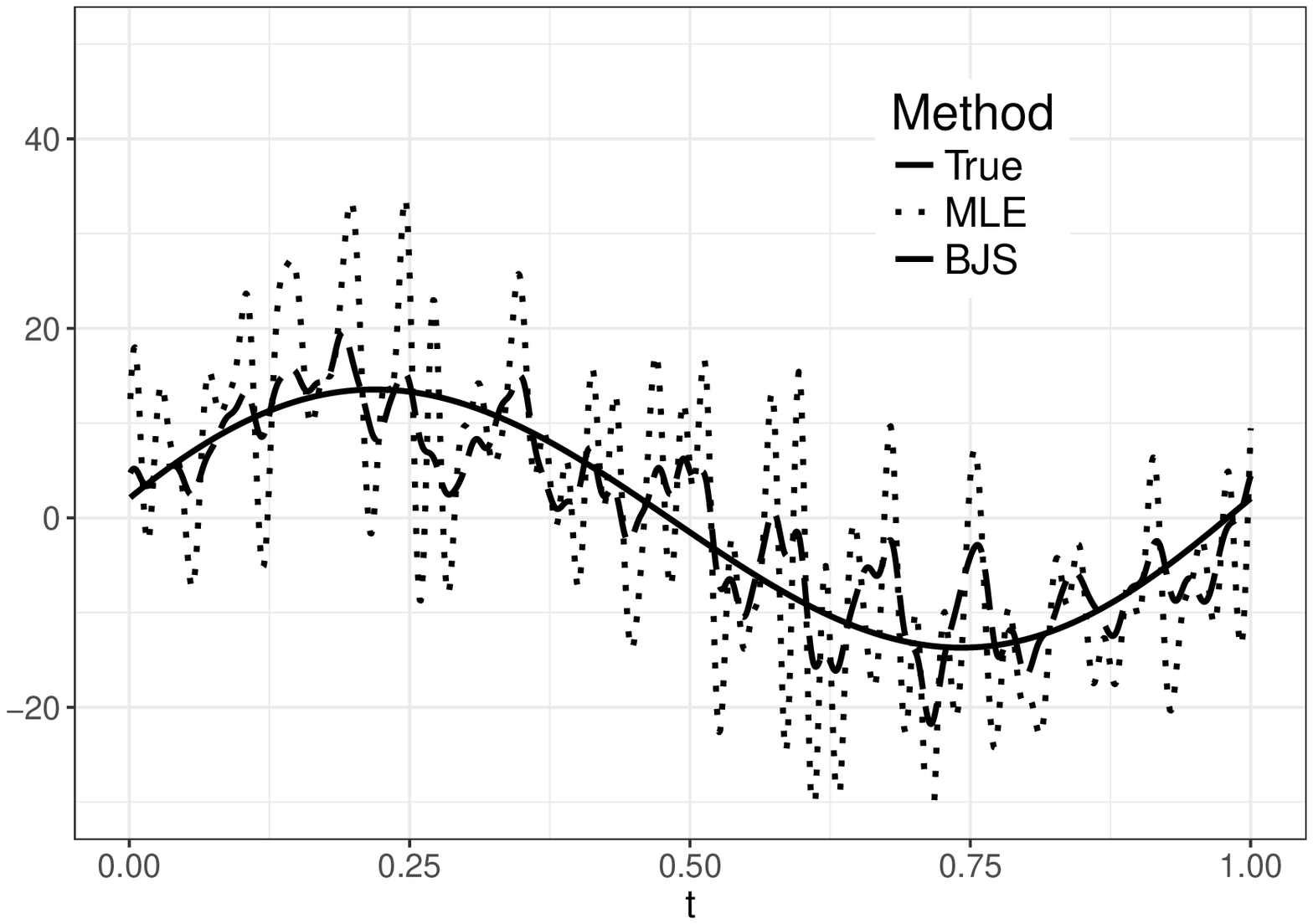}
	\caption{White noise representation of the true parameter and estimators without non-asymptotic adaptation  at $\theta=\theta^{(4)}$ and $B=10$.}
	\label{Nonscaleratio_2}
	\end{center}
	\end{minipage}
	\begin{minipage}{0.4\hsize}
	\begin{center}
		\includegraphics[width=0.90\hsize]{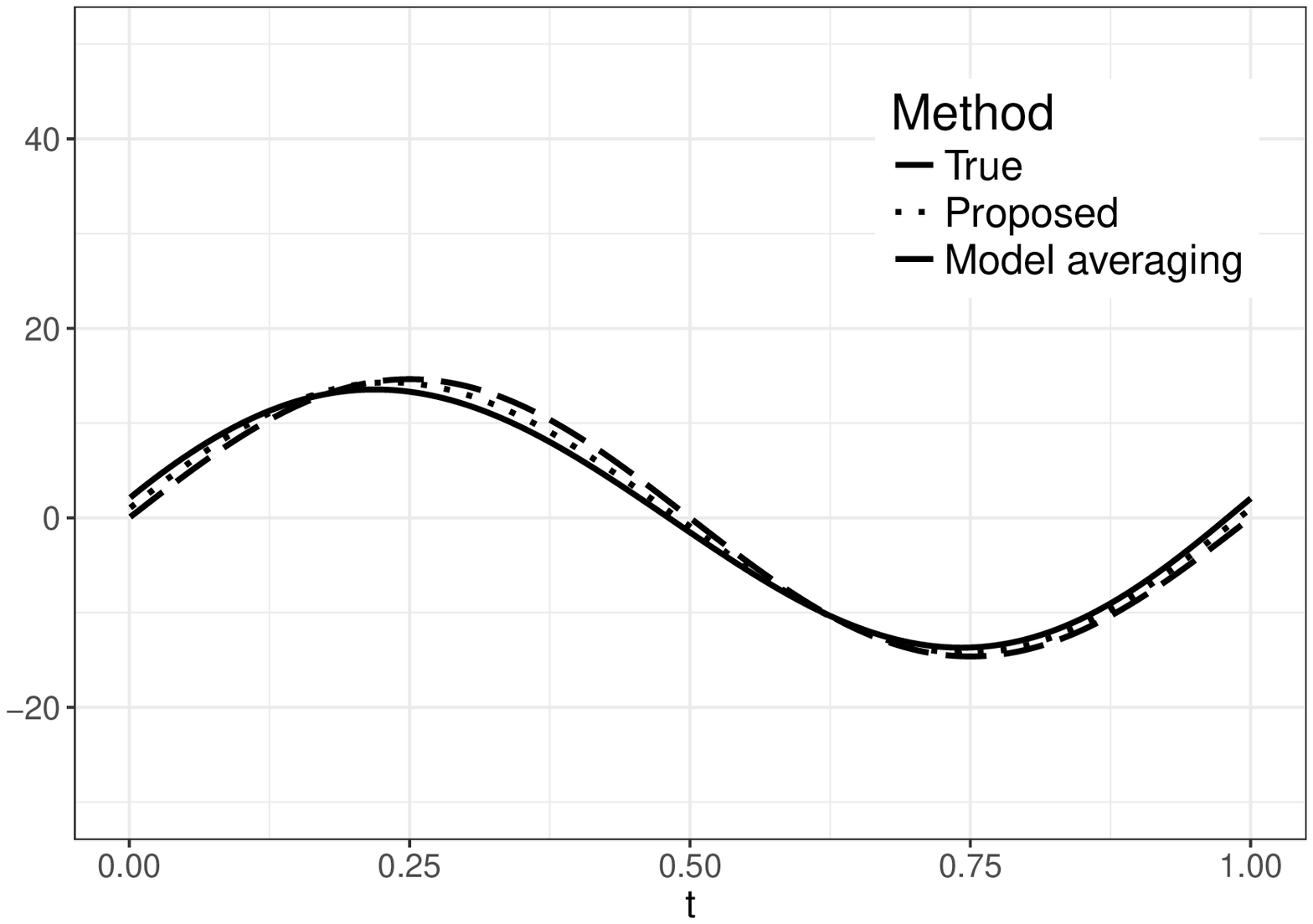}
	\caption{White noise representation of the true signal and estimators without non-asymptotic adaptation  at $\theta=\theta^{(4)}$ and $B=10$.$\quad$}
	\label{Scaleratio_2}
	\end{center}
	\end{minipage}
\end{figure}

Figure \ref{Nonscaleratio_1} shows the true parameter (abbreviated by ``True"),
the maximum likelihood estimator (abbreviated by ``MLE"),
and the blockwise James--Stein estimator (abbreviated by ``BJS")
at $\theta=\theta^{(3)}$ and $B=10$.
Figure \ref{Scaleratio_1} shows the true parameter,
the proposed Bayes estimator (abbreviated by ``Proposed"),
and the model averaging estimator (abbreviated by ``Model averaging")
at $\theta=\theta^{(3)}$ and $B=10$.
Figures \ref{Nonscaleratio_2} and \ref{Scaleratio_2} shows these at $\theta=\theta^{(4)}$ and $B=10$.

Figures \ref{Nonscaleratio_1} and \ref{Nonscaleratio_2}
indicate that
estimators without non-asymptotic adaptation do not work as smoothers when $B$ is relatively large.
Figures \ref{Scaleratio_1} and \ref{Scaleratio_2}
indicate that
estimators with non-asymptotic adaptation detect the true parameter.

\subsection{Comparison with the Gaussian scale mixture prior distribution}

We compare the performance of the proposed Bayes estimator
with that of the Bayes estimator based on the Gaussian scale mixture prior distribution.
The comparison is intended to indicate that the Bayes estimator based on the Gaussian scale mixture prior 
would also be non-asymptotically adaptive
as conjectured in Remark \ref{rem: Another possibility} in Section \ref{sec:nonasymptoticBayesianadaptation}.

\begin{figure}[h]
	\begin{minipage}{0.4\hsize}
	\begin{center}
		\includegraphics[width=0.90\hsize]{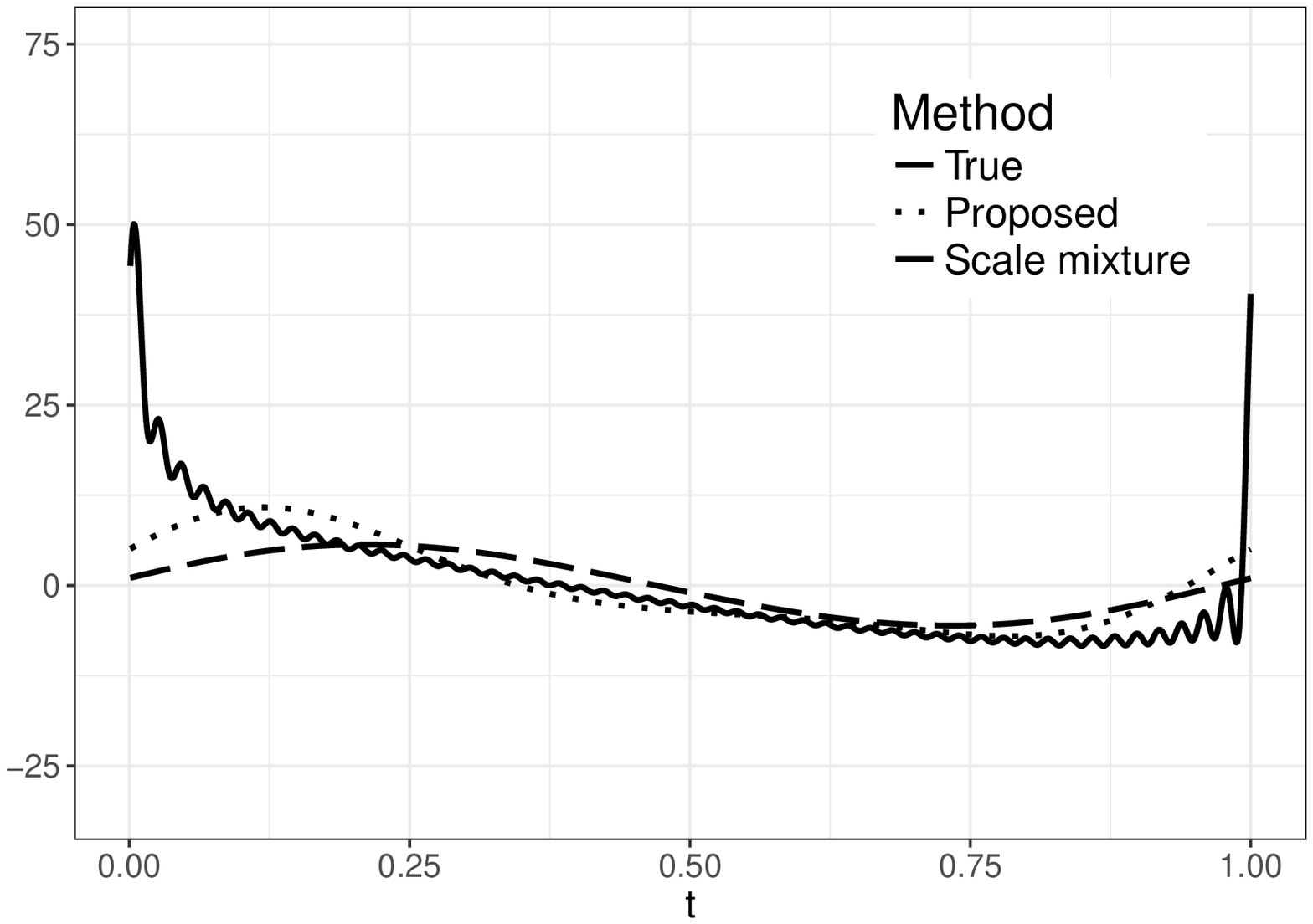}
	\caption{White noise representation of true parameter, the proposed estimator, and the Bayes estimator based on the Gaussian scale mixture prior at $\theta=\theta^{(3)}$ and $B=10$.}
	\label{Scale_1}
	\end{center}
	\end{minipage}
	\begin{minipage}{0.4\hsize}
	\begin{center}
		\includegraphics[width=0.90\hsize]{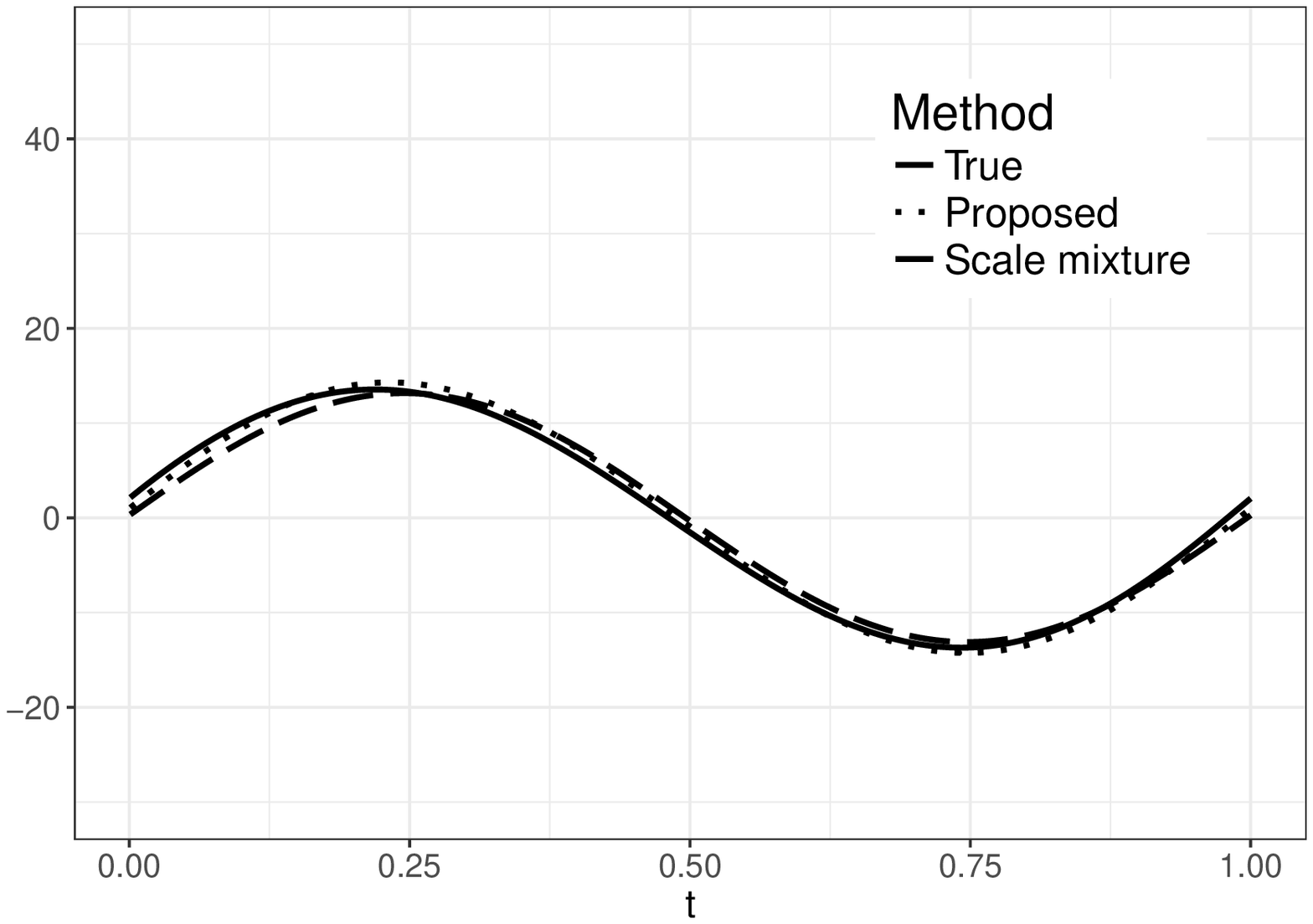}
	\caption{White noise representation of true parameter, the proposed estimator, and the Bayes estimator based on the Gaussian scale mixture prior at $\theta=\theta^{(4)}$ and $B=10$.}
	\label{Scale_2}
	\end{center}
	\end{minipage}
\end{figure}

\begin{figure}[h]
	\begin{minipage}{0.4\hsize}
	\begin{center}
		\includegraphics[width=0.90\hsize]{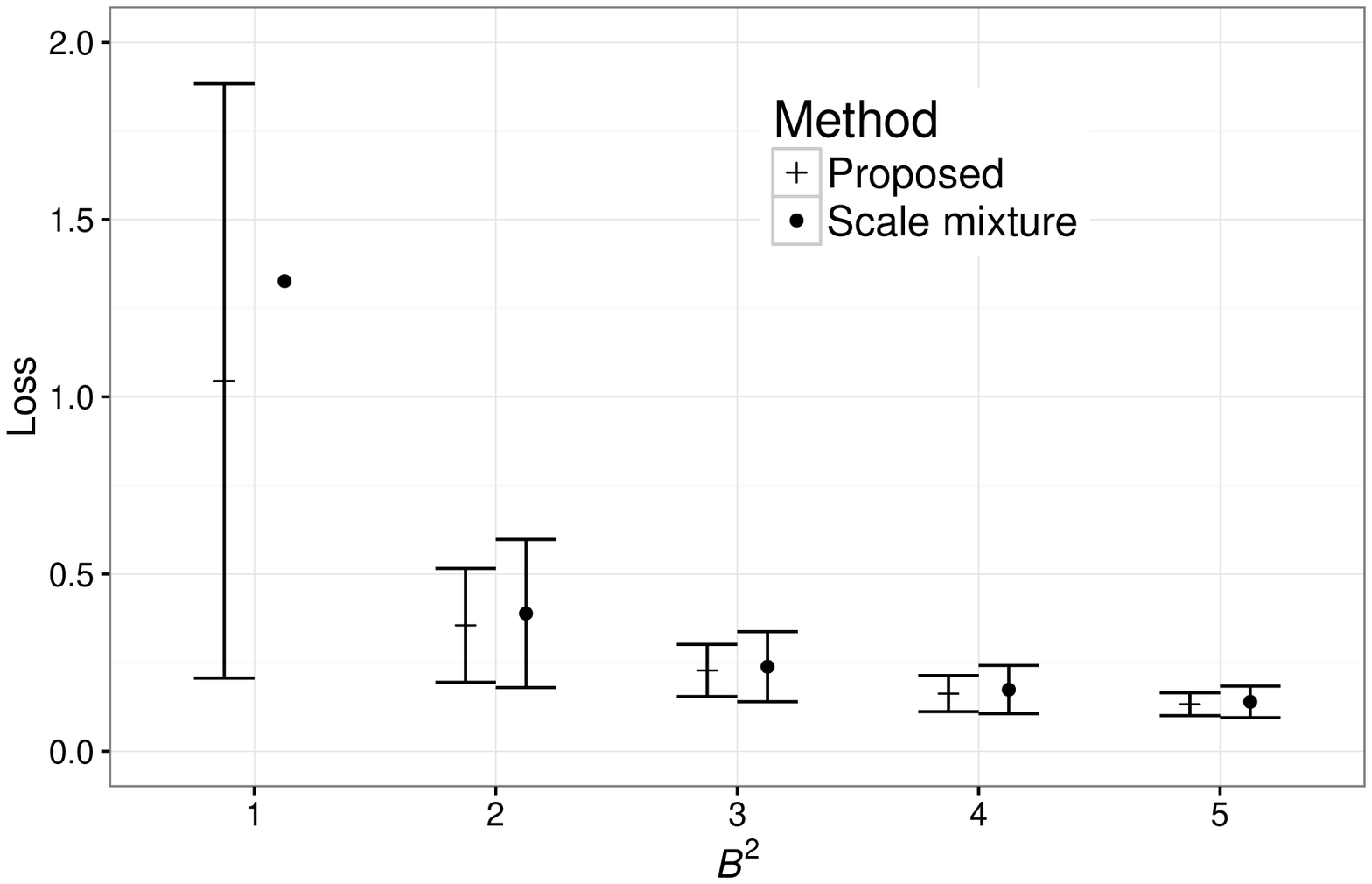}
	\caption{Means of losses with error bars at $\theta=\theta^{(3)}$ in cases with $B^2=1,2,3,4,5$.
	The error bars of the scale mixture estimator at $B^{2}=1$ are omitted because
        the upper bar is outside the range $[0,2]$.
	}
	\label{Loss_Scale_3}
	\end{center}
	\end{minipage}
	\begin{minipage}{0.4\hsize}
	\begin{center}
		\includegraphics[width=0.90\hsize]{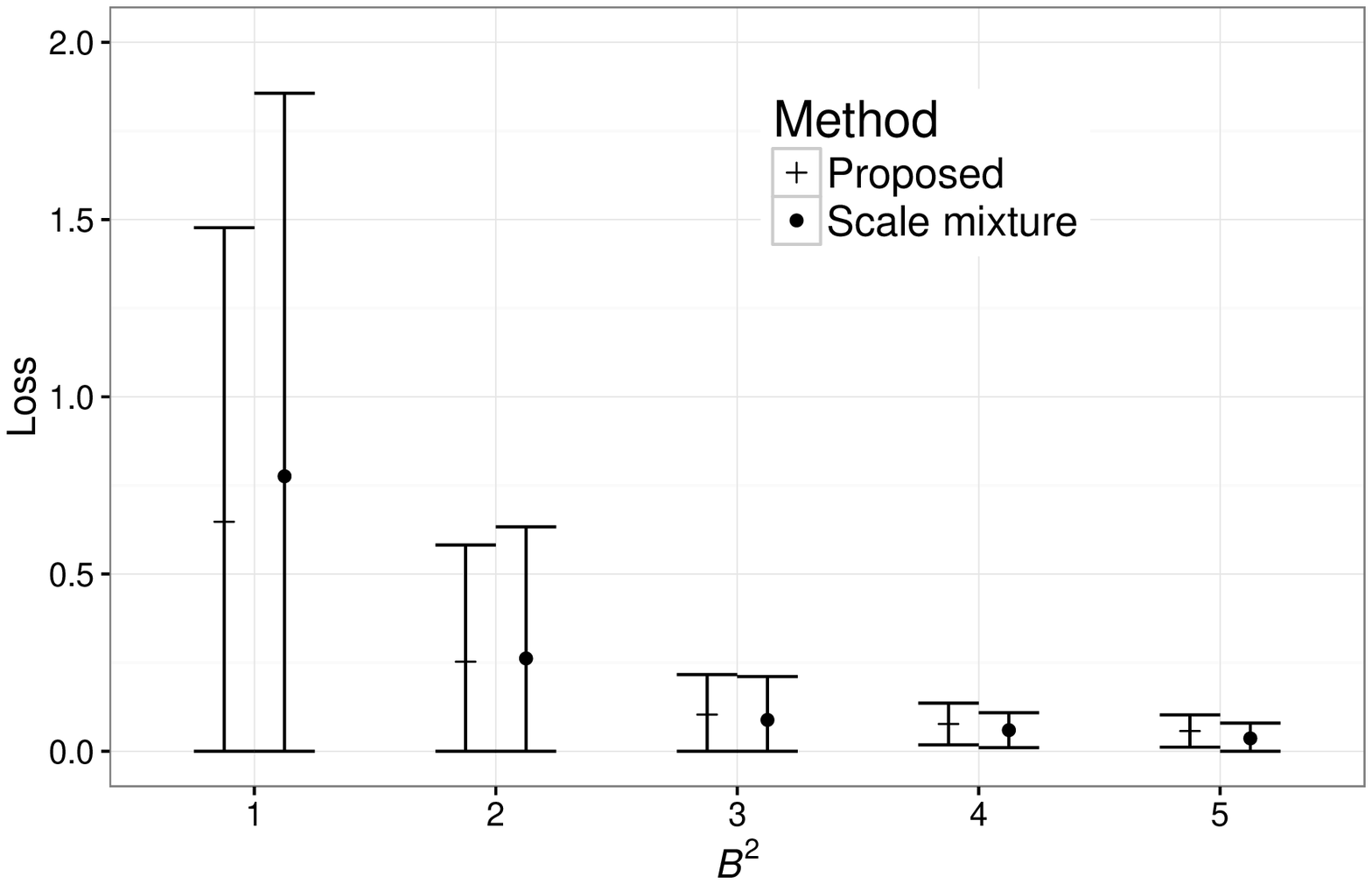}
	\caption{Means of losses with error bars at $\theta=\theta^{(4)}$ in cases with $B^2=1,2,3,4,5$.}
	\label{Loss_Scale_4}
	\end{center}
	\end{minipage}
\end{figure}

Figures \ref{Scale_1} and \ref{Scale_2} are comparisons using the white noise representation.
Figures \ref{Loss_Scale_3} and \ref{Loss_Scale_4} are comparisons using the values of losses at $\theta = \theta^{(3)},\theta^{(4)}$.
The proposed Bayes estimator is abbreviated by ``Proposed," and
the Bayes estimator based on the Gaussian scale mixture prior distribution is abbreviated by ``Scale mixture."

Figures \ref{Scale_1}--\ref{Loss_Scale_4} indicate that 
the performance of the Bayes estimator based on the Gaussian scale mixture prior distribution
is comparable to that of the proposed estimator.

\bibliographystyle{imsart-number}
\bibliography{ScaleRatioMinimax}

\begin{thebibliography}{45}

\bibitem{Akaike(1973)}
\begin{binproceedings}[author]
\bauthor{\bsnm{Akaike},~\bfnm{H.}\binits{H.}}
(\byear{1973}).
\btitle{Information theory and extension of the maximum likelihood principle}.
In \bbooktitle{Proc. 2nd Int. Symp. Info. Theory}
\bpages{\,267--281}.
\end{binproceedings}
\endbibitem

\bibitem{Arbeletal(2013)}
\begin{barticle}[author]
\bauthor{\bsnm{Arbel},~\bfnm{J.}\binits{J.}},
  \bauthor{\bsnm{Gayraud},~\bfnm{G.}\binits{G.}} \AND
  \bauthor{\bsnm{Rousseau},~\bfnm{J.}\binits{J.}}
(\byear{2013}).
\btitle{Bayesian optimal adaptive estimation using a sieve prior}.
\bjournal{Scand. J. Statist.}
\bvolume{40}
\bpages{\,549--570}.
\end{barticle}
\endbibitem

\bibitem{Baraud(2000)}
\begin{barticle}[author]
\bauthor{\bsnm{Baraud},~\bfnm{Y.}\binits{Y.}}
(\byear{2000}).
\btitle{Model selection for regression on a fixed design}.
\bjournal{Probab. Theory Relat. Fields}
\bvolume{117}
\bpages{\,467--493}.
\end{barticle}
\endbibitem

\bibitem{BarronBirgeMassart(1999)}
\begin{barticle}[author]
\bauthor{\bsnm{Barron},~\bfnm{A.}\binits{A.}},
  \bauthor{\bsnm{Birg\'{e}},~\bfnm{L.}\binits{L.}} \AND
  \bauthor{\bsnm{Massart},~\bfnm{P.}\binits{P.}}
(\byear{1999}).
\btitle{Risk bounds for model selection via penalization}.
\bjournal{Probab. Theory Relat. Fields}
\bvolume{113}
\bpages{\,301--413}.
\end{barticle}
\endbibitem

\bibitem{BarronSchervishWasserman(1999)}
\begin{barticle}[author]
\bauthor{\bsnm{Barron},~\bfnm{A.}\binits{A.}},
  \bauthor{\bsnm{Schervish},~\bfnm{M.}\binits{M.}} \AND
  \bauthor{\bsnm{Wasserman},~\bfnm{L.}\binits{L.}}
(\byear{1999}).
\btitle{The consistency of posterior distributions in nonparametric problems}.
\bjournal{Ann. Statist.}
\bvolume{27}
\bpages{\,536--561}.
\end{barticle}
\endbibitem

\bibitem{BelitserandGhosal(2003)}
\begin{barticle}[author]
\bauthor{\bsnm{Belitser},~\bfnm{E.}\binits{E.}} \AND
  \bauthor{\bsnm{Ghosal},~\bfnm{S.}\binits{S.}}
(\byear{2003}).
\btitle{Adaptive {B}ayesian Inference of the mean of an inifinite-dimensional
  normal distribution}.
\bjournal{Ann. Statist.}
\bvolume{31}
\bpages{\,536--559}.
\end{barticle}
\endbibitem

\bibitem{BirgeandMassart(1997)}
\begin{binproceedings}[author]
\bauthor{\bsnm{Birg\'{e}},~\bfnm{L.}\binits{L.}} \AND
  \bauthor{\bsnm{Massart},~\bfnm{P.}\binits{P.}}
(\byear{1997}).
\btitle{From model selection to adaptive estimation}.
In \bbooktitle{Festschrift for Lucien Le Cam: Research Papers in Probability
  and Statistics}
\bpages{\,55--87}.
\end{binproceedings}
\endbibitem

\bibitem{BirgeandMassart(2001)}
\begin{barticle}[author]
\bauthor{\bsnm{Birg\'{e}},~\bfnm{L.}\binits{L.}} \AND
  \bauthor{\bsnm{Massart},~\bfnm{P.}\binits{P.}}
(\byear{2001}).
\btitle{Gaussian model selection}.
\bjournal{J. Eur. Math. Soc.}
\bvolume{3}
\bpages{\,203--268}.
\end{barticle}
\endbibitem

\bibitem{BrownandLow(1996)}
\begin{barticle}[author]
\bauthor{\bsnm{Brown},~\bfnm{L.}\binits{L.}} \AND
  \bauthor{\bsnm{Low},~\bfnm{M.}\binits{M.}}
(\byear{1996}).
\btitle{Asymptotic equivalence of nonparametric regression and white noise}.
\bjournal{Ann. Statist.}
\bvolume{24}
\bpages{\,2384--2398}.
\end{barticle}
\endbibitem

\bibitem{Caietal(2000)}
\begin{btechreport}[author]
\bauthor{\bsnm{Cai},~\bfnm{T.}\binits{T.}},
  \bauthor{\bsnm{Low},~\bfnm{M.}\binits{M.}} \AND
  \bauthor{\bsnm{Zhao},~\bfnm{L.}\binits{L.}}
(\byear{2000}).
\btitle{Sharp adaptive estimation by a blockwise method.}
\btype{Technical Report},
\bpublisher{Wharton School, University of Pennsylvania, Philadelphia}.
\end{btechreport}
\endbibitem

\bibitem{CavalierandTsybakov(2001)}
\begin{barticle}[author]
\bauthor{\bsnm{Cavalier},~\bfnm{L.}\binits{L.}} \AND
  \bauthor{\bsnm{Tsybakov},~\bfnm{A.}\binits{A.}}
(\byear{2001}).
\btitle{Penalized blockwise Stein's method, monotone oracles and sharp adaptive
  estimation}.
\bjournal{Math. Methods of Statist.}
\bvolume{10}
\bpages{\,247--282}.
\end{barticle}
\endbibitem

\bibitem{DalalyanandSalmon(2012)}
\begin{barticle}[author]
\bauthor{\bsnm{Dalalyan},~\bfnm{A.}\binits{A.}} \AND
  \bauthor{\bsnm{Salmon},~\bfnm{J.}\binits{J.}}
(\byear{2012}).
\btitle{Sharp oracle inequalities for aggregation of affine estimators}.
\bjournal{Ann. Statist.}
\bvolume{40}
\bpages{\,2327--2355}.
\end{barticle}
\endbibitem

\bibitem{Efromovich(1999)}
\begin{bbook}[author]
\bauthor{\bsnm{Efromovich},~\bfnm{S.}\binits{S.}}
(\byear{1999}).
\btitle{Nonparametric Curve Estimation}.
\bpublisher{Springer}.
\end{bbook}
\endbibitem

\bibitem{EfromovichandPinsker(1984)}
\begin{barticle}[author]
\bauthor{\bsnm{Efromovich},~\bfnm{S.}\binits{S.}} \AND
  \bauthor{\bsnm{Pinsker},~\bfnm{M.}\binits{M.}}
(\byear{1984}).
\btitle{Learning algorithm for nonparmetric filtering}.
\bjournal{Automation and Remote Control}
\bvolume{11}
\bpages{\,1434--1440}.
\end{barticle}
\endbibitem

\bibitem{Freedman(1999)}
\begin{barticle}[author]
\bauthor{\bsnm{Freedman},~\bfnm{D.}\binits{D.}}
(\byear{1999}).
\btitle{On the {B}ernstein--von {M}ises theorem with infinite-dimensional
  parameters}.
\bjournal{Ann. Statist.}
\bvolume{27}
\bpages{\,1119--1140}.
\end{barticle}
\endbibitem

\bibitem{GaoandZhou(2016)}
\begin{barticle}[author]
\bauthor{\bsnm{Gao},~\bfnm{C.}\binits{C.}} \AND
  \bauthor{\bsnm{Zhou},~\bfnm{H.}\binits{H.}}
(\byear{2016}).
\btitle{Rate exact {B}ayesian adaptation with modified block priors}.
\bjournal{Ann. Statist.}
\bvolume{44}
\bpages{\,318--345}.
\end{barticle}
\endbibitem

\bibitem{Ghosal_Ghosh_vanderVaart(2000)}
\begin{barticle}[author]
\bauthor{\bsnm{Ghosal},~\bfnm{S.}\binits{S.}},
  \bauthor{\bsnm{Ghosh},~\bfnm{J.}\binits{J.}} \AND
  \bauthor{\bparticle{van~der} \bsnm{Vaart},~\bfnm{A.}\binits{A.}}
(\byear{2000}).
\btitle{Convergence rate of posterior distributions}.
\bjournal{Ann. Statist.}
\bvolume{28}
\bpages{\,500--531}.
\end{barticle}
\endbibitem

\bibitem{GhosalLembervanderVaart(2008)}
\begin{barticle}[author]
\bauthor{\bsnm{Ghosal},~\bfnm{S.}\binits{S.}},
  \bauthor{\bsnm{Lember},~\bfnm{J.}\binits{J.}} \AND
  \bauthor{\bparticle{van~der} \bsnm{Vaart},~\bfnm{A.}\binits{A.}}
(\byear{2008}).
\btitle{Nonparametric {B}ayesian model selection and averaging}.
\bjournal{Elec. J. Statist.}
\bvolume{2}
\bpages{\,63--89}.
\end{barticle}
\endbibitem

\bibitem{GhosalandvanderVaart(2007)}
\begin{barticle}[author]
\bauthor{\bsnm{Ghosal},~\bfnm{S.}\binits{S.}} \AND \bauthor{\bparticle{van~der}
  \bsnm{Vaart},~\bfnm{A.}\binits{A.}}
(\byear{2007}).
\btitle{Convergence rates of posterior distributions for noniid observations}.
\bjournal{Ann. Statist.}
\bvolume{35}
\bpages{\,192--223}.
\end{barticle}
\endbibitem

\bibitem{GineandNickl(2016)}
\begin{bbook}[author]
\bauthor{\bsnm{Gin\'e},~\bfnm{E.}\binits{E.}} \AND
  \bauthor{\bsnm{Nickl},~\bfnm{R.}\binits{R.}}
(\byear{2016}).
\btitle{Mathematical foundations of infinite-dimensional statistical models}.
\bpublisher{Cambridge University Press}.
\end{bbook}
\endbibitem

\bibitem{Hartigan(2002)}
\begin{btechreport}[author]
\bauthor{\bsnm{Hartigan},~\bfnm{J.}\binits{J.}}
(\byear{2002}).
\btitle{Bayesian Regression Using Akaike Priors}
\btype{Technical Report},
\bpublisher{New Haven, CT, Yale University}.
\end{btechreport}
\endbibitem

\bibitem{HoffmannRousseauSchmidt-Hieber(2015)}
\begin{barticle}[author]
\bauthor{\bsnm{Hoffmann},~\bfnm{M.}\binits{M.}},
  \bauthor{\bsnm{Rousseau},~\bfnm{J.}\binits{J.}} \AND
  \bauthor{\bsnm{Schmidt-Hieber},~\bfnm{J.}\binits{J.}}
(\byear{2015}).
\btitle{On Adaptive posterior concentration rates}.
\bjournal{Ann. Statist.}
\bvolume{43}
\bpages{\,2259--2295}.
\end{barticle}
\endbibitem

\bibitem{Huang(2004)}
\begin{barticle}[author]
\bauthor{\bsnm{Huang},~\bfnm{T.}\binits{T.}}
(\byear{2004}).
\btitle{Convergence rates for posterior distributions and adaptive estimation}.
\bjournal{Ann. Statist.}
\bvolume{32}
\bpages{\,1556--1593}.
\end{barticle}
\endbibitem

\bibitem{Johannes_Simoni_Schenk(2016)}
\begin{binproceedings}[author]
\bauthor{\bsnm{Johannes},~\bfnm{J.}\binits{J.}},
  \bauthor{\bsnm{Schenk},~\bfnm{R.}\binits{R.}} \AND
  \bauthor{\bsnm{Simoni},~\bfnm{A.}\binits{A.}}
(\byear{2014}).
\btitle{Adaptive {B}ayesian estimation in {G}aussian sequence space models}.
In \bbooktitle{Contributions in infinite-dimensional statistics and related
  topics}
\bpages{\,167--172}.
\end{binproceedings}
\endbibitem

\bibitem{KnapikSzabovanderVaartvanZanten(2016)}
\begin{barticle}[author]
\bauthor{\bsnm{Knapik},~\bfnm{B.}\binits{B.}},
  \bauthor{\bsnm{Szab\'o},~\bfnm{B.}\binits{B.}}, \bauthor{\bparticle{van~der}
  \bsnm{Vaart},~\bfnm{A.}\binits{A.}} \AND \bauthor{\bparticle{van}
  \bsnm{Zanten},~\bfnm{H.}\binits{H.}}
(\byear{2016}).
\btitle{Bayes procedures for adaptive inference in inverse problems for the
  white noise model}.
\bjournal{Probab. Theory Relat. Fields}
\bvolume{164}
\bpages{\,771--813}.
\end{barticle}
\endbibitem

\bibitem{KnapikvanderVaartvanZanten(2011)}
\begin{barticle}[author]
\bauthor{\bsnm{Knapik},~\bfnm{B.}\binits{B.}}, \bauthor{\bparticle{van~der}
  \bsnm{Vaart},~\bfnm{A.}\binits{A.}} \AND \bauthor{\bparticle{van}
  \bsnm{Zanten},~\bfnm{H.}\binits{H.}}
(\byear{2011}).
\btitle{Bayesian inverse problems with {G}aussian priors}.
\bjournal{Ann. Statist.}
\bvolume{39}
\bpages{\,2626--2657}.
\end{barticle}
\endbibitem

\bibitem{LeungandBarron(2006)}
\begin{barticle}[author]
\bauthor{\bsnm{Leung},~\bfnm{G.}\binits{G.}} \AND
  \bauthor{\bsnm{Barron},~\bfnm{A.}\binits{A.}}
(\byear{2006}).
\btitle{Information Theory and Mixing Least-Squares Regressions}.
\bjournal{IEEE tran. on INFOR. THEORY}
\bvolume{52}
\bpages{\,3396--3410}.
\end{barticle}
\endbibitem

\bibitem{Mallows(1973)}
\begin{barticle}[author]
\bauthor{\bsnm{Mallows},~\bfnm{C.}\binits{C.}}
(\byear{1973}).
\btitle{Some comments on $C_{p}$}.
\bjournal{Technometrics}
\bvolume{15}
\bpages{\,661--675}.
\end{barticle}
\endbibitem

\bibitem{Massart(2007)}
\begin{bbook}[author]
\bauthor{\bsnm{Massart},~\bfnm{P.}\binits{P.}}
(\byear{2007}).
\btitle{Concentration Inequalities and Model Selection: Ecole d'Et\'{e} de
  Probabilit\'{e}s de Saint-Flour XXXIII-2003}.
\bpublisher{Springer}.
\end{bbook}
\endbibitem

\bibitem{PetroneRousseauScricciolo(2014)}
\begin{barticle}[author]
\bauthor{\bsnm{Petrone},~\bfnm{S.}\binits{S.}},
  \bauthor{\bsnm{Rousseau},~\bfnm{J.}\binits{J.}} \AND
  \bauthor{\bsnm{Scricciolo},~\bfnm{C.}\binits{C.}}
(\byear{2014}).
\btitle{{B}ayes and empirical {B}ayes: do they merge?}
\bjournal{Biometrika}
\bvolume{101}
\bpages{\,285--302}.
\end{barticle}
\endbibitem

\bibitem{Pinsker(1980)}
\begin{barticle}[author]
\bauthor{\bsnm{Pinsker},~\bfnm{M.}\binits{M.}}
(\byear{1980}).
\btitle{Optimal filtering of square integrable signals in {G}aussian white
  noise}.
\bjournal{Problems Inform. Transmission}
\bvolume{16}
\bpages{\,120--133}.
\end{barticle}
\endbibitem

\bibitem{RasmussenandWilliams(2005)}
\begin{bbook}[author]
\bauthor{\bsnm{Rasmussen},~\bfnm{C.}\binits{C.}} \AND
  \bauthor{\bsnm{Williams},~\bfnm{K.}\binits{K.}}
(\byear{2005}).
\btitle{Gaussian Processes for Machine Learning}.
\bpublisher{the MIT Press}.
\end{bbook}
\endbibitem

\bibitem{Ray(2013)}
\begin{barticle}[author]
\bauthor{\bsnm{Ray},~\bfnm{K.}\binits{K.}}
(\byear{2013}).
\btitle{Bayesian inverse problems with non-conjugate priors}.
\bjournal{Elec. J. Statist.}
\bvolume{7}
\bpages{\,2516--2549}.
\end{barticle}
\endbibitem

\bibitem{RousseauandSzabo(2017)}
\begin{barticle}[author]
\bauthor{\bsnm{Rousseau},~\bfnm{J.}\binits{J.}} \AND
  \bauthor{\bsnm{Szab\'o},~\bfnm{B.}\binits{B.}}
(\byear{2017}).
\btitle{Asymptotic behaviour of the empirical {B}ayes posterior associated to
  maximum marginal likelihood estimator}.
\bjournal{Ann. Statist.}
\bvolume{45}
\bpages{\,833--865}.
\end{barticle}
\endbibitem

\bibitem{Scricciolo(2006)}
\begin{barticle}[author]
\bauthor{\bsnm{Scricciolo},~\bfnm{C.}\binits{C.}}
(\byear{2006}).
\btitle{Convergence rates for {B}ayesian density estimation of
  infinite-dimensional exponential families}.
\bjournal{Ann. Statist.}
\bvolume{34}
\bpages{\,2897--2920}.
\end{barticle}
\endbibitem

\bibitem{ShenandGhosal(2015)}
\begin{barticle}[author]
\bauthor{\bsnm{Shen},~\bfnm{W.}\binits{W.}} \AND
  \bauthor{\bsnm{Ghosal},~\bfnm{S.}\binits{S.}}
(\byear{2015}).
\btitle{Adaptive Bayesian Procedures Using Random Series Priors}.
\bjournal{Scand. J. Statist.}
\bvolume{42}
\bpages{\,1194--1213}.
\end{barticle}
\endbibitem

\bibitem{ShenandWasserman(2001)}
\begin{barticle}[author]
\bauthor{\bsnm{Shen},~\bfnm{X.}\binits{X.}} \AND
  \bauthor{\bsnm{Wasserman},~\bfnm{L.}\binits{L.}}
(\byear{2001}).
\btitle{Rate of Convergence of posterior distributions}.
\bjournal{Ann. Statist.}
\bvolume{29}
\bpages{\,687--714}.
\end{barticle}
\endbibitem

\bibitem{Stein(1973)}
\begin{binproceedings}[author]
\bauthor{\bsnm{Stein},~\bfnm{C.}\binits{C.}}
(\byear{1973}).
\btitle{Estimation of the mean of a multivariate normal distribution}.
In \bbooktitle{Proc. Prague Symp. Asmptotic Statistics}
\bpages{\,345--381}.
\end{binproceedings}
\endbibitem

\bibitem{Suzuki(2012)}
\begin{binproceedings}[author]
\bauthor{\bsnm{Suzuki},~\bfnm{T.}\binits{T.}}
(\byear{2012}).
\btitle{{PAC}-{B}ayesian Bound for {G}aussian Process Regression and Multiple
  Kernel Additive Model}.
In \bbooktitle{25th annual {C}onference {O}n {L}earning {T}heory}
\bpages{\,8.1--8.20}.
\end{binproceedings}
\endbibitem

\bibitem{Szabo_vanderVaart_vanZanten(2013)}
\begin{barticle}[author]
\bauthor{\bsnm{Szab\'o},~\bfnm{B.}\binits{B.}}, \bauthor{\bparticle{van~der}
  \bsnm{Vaart},~\bfnm{A.}\binits{A.}} \AND \bauthor{\bparticle{van}
  \bsnm{Zanten},~\bfnm{H.}\binits{H.}}
(\byear{2013}).
\btitle{Empirical {B}ayes scaling of {G}aussian priors in the white noise
  model}.
\bjournal{Elec. J. Statist.}
\bvolume{7}
\bpages{\,991--1018}.
\end{barticle}
\endbibitem

\bibitem{Tsybakov(2009)}
\begin{bbook}[author]
\bauthor{\bsnm{Tsybakov},~\bfnm{A.}\binits{A.}}
(\byear{2009}).
\btitle{Introduction to Nonparametric Estimation}.
\bpublisher{Springer}.
\end{bbook}
\endbibitem

\bibitem{vanderVaartandvanZanten(2009)}
\begin{barticle}[author]
\bauthor{\bparticle{van~der} \bsnm{Vaart},~\bfnm{A.}\binits{A.}} \AND
  \bauthor{\bparticle{van} \bsnm{Zanten},~\bfnm{H.}\binits{H.}}
(\byear{2009}).
\btitle{Adaptive {B}ayesian Estimation Using A {G}aussian Random Field with
  Inverse {G}amma Bandwidth}.
\bjournal{Ann. Statist.}
\bvolume{37}
\bpages{\,2655--2675}.
\end{barticle}
\endbibitem

\bibitem{Wasserman(2006)}
\begin{bbook}[author]
\bauthor{\bsnm{Wasserman},~\bfnm{L.}\binits{L.}}
(\byear{2006}).
\btitle{All of Nonparametric Statistics}.
\bpublisher{Springer}.
\end{bbook}
\endbibitem

\bibitem{Yang(2005)}
\begin{barticle}[author]
\bauthor{\bsnm{Yang},~\bfnm{Y.}\binits{Y.}}
(\byear{2005}).
\btitle{Can the strengths of {AIC} and {BIC} be shared? {A} conflict between
  model indentification and regression estimation}.
\bjournal{Biometrika}
\bvolume{92}
\bpages{\,937--950}.
\end{barticle}
\endbibitem

\bibitem{Zhao(2000)}
\begin{barticle}[author]
\bauthor{\bsnm{Zhao},~\bfnm{L.}\binits{L.}}
(\byear{2000}).
\btitle{{B}ayesian aspects of some nonparametric problems}.
\bjournal{Ann. Statist.}
\bvolume{28}
\bpages{\,532--552}.
\end{barticle}
\endbibitem

\end{thebibliography}
\end{document}